\documentclass{elsarticle}
\usepackage[utf8]{inputenc}

\usepackage{xr}
\usepackage{amssymb}
\usepackage{amsmath}
\usepackage{amsthm}

\usepackage{pgf,tikz,tkz-graph}
\usetikzlibrary{arrows,shapes}

\numberwithin{equation}{section}
\theoremstyle{plain}

\newtheorem{thm}{Theorem}[section]
\newtheorem{lemma}[thm]{Lemma}
\newtheorem{prop}[thm]{Proposition}

\newtheorem{claim}[thm]{Claim}
\newtheorem*{lemma*}{Lemma}
\newtheorem*{prop*}{Proposition}
\newtheorem*{claim*}{Claim}
\newtheorem*{cor*}{Corollary}
\newtheorem*{thm*}{Theorem}

\theoremstyle{definition}
\newtheorem{defn}[thm]{Definition}

\newtheorem{prob}[thm]{Problem}
\newtheorem{questi}[thm]{Question}
\newtheorem{conject}[thm]{Conjecture}
\newtheorem{rmk}[thm]{Remark}

\newtheorem*{defn*}{Definition}
\newtheorem*{rmk*}{Remark}

\usepackage{xspace}

\def\mindeg{{\rm mindeg}}
\def\dx{\,dx}
\def\dy{\,dy}
\def\dz{\,dz}
\def\du{\,du}
\def\dt{\,dt}

\newcommand{\N}{\mathbb{N}}
\newcommand{\R}{\mathbb{R}}
\newcommand{\rplus}{\mathbb{R}_{\ge 0}}
\newcommand{\Z}{\mathbb{Z}}
\newcommand{\Sphere}{\mathbb{S}}
\DeclareMathOperator*{\chip}{Chip}
\DeclareMathOperator*{\id}{id}

\begin{document}

\title{The devil's staircase for chip-firing on random graphs and on graphons}

\author[renyi]{Viktor Kiss\fnref{viktor}}
\ead{kiss.viktor@renyi.hu}
\fntext[viktor]{Partially supported by NSF grant DMS-1455272, and by the National Research, Development and Innovation Office -- NKFIH, grants no. 104178, 124749, 129211, and 128273.}

\author[cornell]{Lionel Levine\fnref{lionel}}
\ead{levine@math.cornell.edu}
\fntext[lionel]{Partially supported by NSF grant DMS-1455272.}

\author[tki]{Lilla T\'othm\'er\'esz\fnref{lilla}}
\ead{tmlilla@cs.elte.hu}
\fntext[lilla]{Partially supported by NSF grant DMS-1455272, and by the National Research, Development and Innovation Office -- NKFIH, grants no. 128673, and 132488.}

\address[renyi]{Alfr\'ed R\'enyi Institute of Mathematics, Re\'altanoda u. 13--15, H-1053 Budapest, Hungary}

\address[cornell]{Cornell University, Ithaca, New York 14853-4201, USA}

\address[tki]{MTA-ELTE Egerv\'ary Research Group, P\'azm\'any P\'eter s\'et\'any 1/C, Budapest, Hungary}

\date{}

\begin{abstract}
    We study the behavior of the activity of the parallel chip-firing upon increasing the number of chips on an Erd\H os--R\'enyi random graph. We show that in various situations the resulting activity diagrams converge to a devil's staircase as we increase the number of vertices. Our method is to generalize the parallel chip-firing to graphons, and to prove a continuity result for the activity. We also show that the activity of a chip configuration on a graphon does not necessarily exist, but it does exist for every chip configuration on a large class of graphons.
\end{abstract}


\maketitle

\section{Introduction}
\label{s:intro}

In this paper, we study the behavior of the activity of the parallel chip-firing upon increasing the number of chips in the system. Numerical experiments of Bagnoli, Cecconi, Flammini, and Vespignani \cite{bagnoli} suggested that for planar grids, upon increasing the number of chips in the system, the activity asymptotically increases as a Devil's staircase. Later, Levine \cite{Lionel_parallel} proved a similar statement for complete graphs, i.e., if we take a sequence of complete graphs whose size tends to infinity, and a sequence of chip configurations on them that converge in a certain sense, then the activity diagrams tend to a Devil's staircase. 
In this paper, we prove analogous statements in various situations for sequences of Erd\H os--R\'enyi random graphs. Our method is to generalize the parallel chip-firing to graphons, and then to prove a continuity theorem for the activity. Levine's results can be interpreted as a Devil's staircase result for the constant graphon. Using our continuity theorem, we can handle the case of sequences of graphons converging to a constant graphon. 

\subsection{Preliminaries}    

Let $G$ be a graph with vertex set $V(G)$ and edge set $E(G)$. We will denote the number of edges connecting vertex $u$ and $v$ by $e_G(u,v)$, and the degree of a vertex $v$ by $\deg_G(v)$. We will often consider Erd\H os--R\'enyi random graphs. By $G(n,p)$ we denote the random graph on $n$ vertices, where each edge is present independently with probability $p$.

For a graph $G$, a \emph{chip configuration} assigns to each vertex a non-negative amount of chips. Hence a chip configuration is a function $\sigma : V(G) \to \rplus$, where $\rplus = \{x \in \R : x \ge 0\}$. In the literature, a chip-configuration is usually considered to be integer-valued, but since we will be interested in the change of dynamics as we gradually increase the amount of chips, we choose to allow nonintegrality. For two chip configurations $\sigma$ and $\sigma'$, $\sigma \geq \sigma'$ means that $\sigma(v) \geq \sigma'(v)$ for each vertex $v$.

Firing a node $v$ means that the fired vertex passes a chip along each edge incident to it, i.e., the chip configuration $\sigma$ gets modified to 
\begin{equation*}
\begin{array}{cl}
\sigma(u)+e_G(u,v) & \text{if $u\neq v$}, \\
\sigma(u) - \deg_G(v) & \text{if $u=v$}.
\end{array}
\end{equation*}

During a step of the parallel chip-firing, each 
vertex $v$ of $G$ fires  $f(v) = \Big\lfloor \frac{\sigma(v)}{\deg_G(v)}\Big\rfloor$ times, where we call $f = f(G, \sigma)$ the \emph{firing vector} of $\sigma$. The resulting  configuration is

\begin{equation*}
  U\sigma(v) =
    \sigma(v) - \deg_G(v)f(v)  + \sum_{u\in V(G)} f(u) e_G(u,v). 
\end{equation*}

We will denote by $U^n \sigma$ for $n\in \mathbb{N}$ the chip configuration after $n$ steps of the parallel chip-firing. 

We denote by $u_n(v) = u_n(G, \sigma)(v)$ the number of times $v$ fired during the first $n$ turns, and call this function the \emph{odometer}, that is, $u_n(G, \sigma)(v)=\sum_{k=0}^{n - 1} f(G, U^k\sigma)(v)$.

It is easy to see that a parallel chip-firing started from the configuration $\sigma$ on a graph $G$ eventually enters a periodic state, and if $G$ is connected then each vertex fires the same number of times in a period. Hence, $\lim_{n\to\infty} \frac{u_n(v)}{n}$ exists and is the same for each $v \in V(G)$. We call this quantity the \emph{activity} of $\sigma$ and denote it by $a(G, \sigma)$. 

We will be interested in the way the activity changes when we add a small amount of chips to each node. The \emph{activity diagram} of $G$ and $\sigma$ is $s(y) = s(G,\sigma)(y) = a(G,\sigma+y\cdot \deg_G)$. Numerical experiments of \cite{bagnoli} suggested that for planar grids of growing size, the activity diagrams tend to a Devil's staircase. We will call a function $c : [a, b] \to [0, 1]$ a \emph{Devil's staircase}, if $c(a) = 0$, $c(b) = 1$, it is continuous, nondecreasing, but locally constant on an open dense set. Levine \cite{Lionel_parallel} proved such a phenomenon for complete graphs, i.e. that if we take a sequence of complete graphs whose size tends to infinity, and a sequence of chip configurations on them that converge in a certain sense, then (with a mild assumption on the limiting chip configuration) the activity diagrams tend to a Devil's staircase. We take this analysis further, and are able to handle the case of sequences of Erd\H os--R\'enyi random graphs.

\subsection{Results}

To analyze the activity diagram on a (dense) graph, we use the theory of graphons. In order to do so, we introduce parallel chip-firing on graphons. On a graphon, the parallel chip-firing is not necessarily eventually periodic, and the activity of a chip configuration might not exist (see Proposition \ref{p:example for activity to not exist}). However, we show that if there is a lower bound on the degrees in the graphon then the activity exists for any chip configuration (see Theorem \ref{t:activity_exists_with_FDC}).

We show a continuity theorem for the activity. This theorem says that for a graphon with a lower bound on the degrees, and a chip configuration that is compatible with the graphon in a mild sense, if another graphon with a lower bound on the degrees is close in cut distance, and the chip configurations are close to each other in the $L^1$ distance, then the activities are also close to each other.
For a precise statement, see Theorem \ref{thm:robustness}. Using the method of \cite{Lionel_parallel}, we show in Theorem \ref{thm:C_p_Devils_staircase} that with some mild assumptions on the chip configuration $\sigma$, for the constant $p$ graphon $C_p$, the activity diagram $s(C_p,\sigma)$ is a Devil's staircase. Combining this result with the continuity theorem, we are able to prove Theorem \ref{t:ER_Devil} that gives a condition on a sequence of Erd\H os--Rényi random graphs and chip configurations that the activity diagrams converge to a Devil's staircase. Finally, we give a concrete example for a one-parameter family of random chip configurations, where the activities tend to a Devil's staircase, see Theorem \ref{t:geometric_staircase}.

\section{Parallel chip-firing on graphons}
\label{ss:parallel cf on graphons}

We now review the notion of graphons defined by Lov\'asz and Szegedy \cite{Lovasz-Szegedy} as limits of sequences of dense graphs, then introduce parallel chip-firing on them. We follow \cite{Lovasz_large_graphs} in introducing graphons. See \cite{Lovasz_large_graphs} for more information about graphons. 

A \emph{graphon} is a Lebesgue measurable function $W : [0, 1]^2 \to [0, 1]$ which is \emph{symmetric}, that is, $W(x, y) = W(y, x)$ for all $x, y \in [0, 1]$. We can think of a graphon as a generalized 
graph: the vertex set of $W$ is the unit interval $[0, 1]$, and instead of specifying whether two vertices, $x$ and $y$ are connected in $W$ or not, we have a real number $W(x, y) = W(y, x)$ describing how well they are connected.

Here, and everywhere else where we do not specify the measure, we mean the Lebesgue measure, that we denote by $\lambda$.

Note that for every (labeled) graph $G$ with vertices $\{v_1, v_2, \dots, v_n\}$ one can construct a corresponding graphon $W_G$ the following way: partition $[0, 1]$ into $n$ measurable sets $A_1, \dots, A_n$ with $\lambda(A_1) = \lambda(A_2) = \dots = \lambda(A_n)$. 
Then set $W_G(x, y) = 1$ if $x \in A_i$ and $y \in A_j$ with $(v_i, v_j) \in E(G)$, and $W_G(x, y) = 0$ otherwise. 

The \emph{degree} of a vertex $x$ of $W$ is $\deg_W(x) = \int_0^1 W(x, y) \dy$. Note that the degree is well-defined for almost all $x\in[0,1]$. We define
$$
\mindeg(W)=\inf\{\varepsilon: \lambda(\{x: \deg_W(x)\geq \varepsilon\})>0 \}.
$$ 
We use the notation $\deg_{W}(x,A)=\int_A W(x, y) \dy$

To define the convergence of graphon sequences, a notion of distance of graphons is needed. It turns out that the right notion is the cut distance of graphons.

\begin{defn}The \emph{(labeled) cut distance} of the graphons $U$ and $W$ is defined by
	$$\mathrm{d}_\Box(U,W)=\sup_{S,T\subseteq [0,1]} \left|\int_S \int_T U(x,y)-W(x,y)\dy\dx\right|.$$
\end{defn}
The labeled cut distance corresponds to comparing the similarity of two graphons when identifying vertices of the same label. 
The \emph{unlabeled cut distance} corresponds to the case where we want to find the best identification of the two vertex sets:  $\delta_{\Box}(U,W)=\inf_{\varphi}\mathrm{d}_\Box(U,W^{\varphi})$ where $\varphi$ runs over the invertible measure preserving transformations of $[0,1]$ to itself, and $W^{\varphi}(x,y)=W(\varphi(x),\varphi(y))$ (see \cite[Subsection 8.2.2]{Lovasz_large_graphs}). For a (labeled) graph $G$ and a graphon $W$, we use the notation $\mathrm{d}_\Box(G, W) = \mathrm{d}_\Box(W_G, W)$, and similarly for $\delta_\Box$. $\delta_\Box$ is a pseudometric on the space of graphons, and by factorizing with the graphons at zero unlabeled cut distance, one obtains a compact metric space. 

In our applications, we consider sequences of Erdős--Rényi graphs $G(n, p)$, $n=1, 2, \dots$. Such a sequence is known to converge to the constant $p$ graphon $C_p$ with probability $1$ in the distance $\delta_\Box$.
Since $C_p^\varphi = C_p$ for any invertible, measure preserving transformation $\varphi$,  $\delta_\Box(W, C_p) = \mathrm{d}_\Box(W, C_p)$ for any graphon $W$. This fact, and the simpler formalization are the reasons that in this paper, we use the labeled cut distance as a metric on graphons. 

The definition of the parallel chip-firing on graphons is analogous to that on finite graphs. A \emph{chip configuration} on a graphon is an (almost everywhere) non-negative function $\sigma \in L^1([0, 1])$. We denote the set of chip configurations on a graphon $W$ by $\chip(W)$, and use $\|\sigma\|_1$ to denote the $L^1$ norm of a chip configuration $\sigma$.   
For a given chip configuration $\sigma$ on a graphon $W$, the parallel update rule is defined similarly as in the case of finite graphs. Let the \emph{firing vector} of $\sigma$ be
\begin{equation*}
  f(x) = f(W, \sigma)(x) = \left\{ \begin{array}{cl}
    \left\lfloor \frac{\sigma(x)}{\deg_W(x)}\right\rfloor & \text{if $\deg_W(x) > 0$} \\
    0 & \text{if $\deg_W(x) = 0$},
  \end{array}\right.
\end{equation*}
then we can define the update rule by 
\begin{equation*}
  U\sigma(x) = \sigma(x) - \deg_W(x)f(x)  + \int_0^1 f(y) W(x,y) \dy.
\end{equation*}
One can easily see that $f(W, \sigma)$ is defined almost everywhere, and it follows from the following claim that $U\sigma(x)$ is finite almost everywhere. 
\begin{claim}\label{c:L1_norm_stays_the_same_after_firing}
  For a arbitrary graphon $W$ and $\sigma \in \chip(W)$, $U\sigma \in \chip(W)$ and $\|U\sigma\|_1 = \|\sigma\|_1$.  
\end{claim}
\begin{proof}
  From the definition of the firing vector, $\sigma(x) - \deg_W(x) f(x) \ge 0$, hence $U\sigma$ is non-negative, and 
  \begin{equation*}
  \begin{split}
    \|U\sigma\|_1 &= \int|\sigma(x) - \deg_W(x) f(x)| \dx + \int \int f(y) W(x, y) \dy \dx \\
                  &= \int\sigma(x) \dx  - \int \deg_W(x) f(x) \dx + \int \int f(y) W(x, y) \dx \dy \\
                  &= \int\sigma(x) \dx  - \int \deg_W(x) f(x) \dx + \int f(y) \deg_W(y) \dy = \|\sigma\|_1,
  \end{split}
  \end{equation*}
  where we used Fubini's theorem for non-negative functions to interchange the integrals. 
\end{proof}

As in the case of finite graphs, the odometer $u_n(x) = u_n(W, \sigma)(x)$ denotes the number of times $x$ fired during the first $n$ turns, i.e.,
$$u_n(x) = u_n(W, \sigma)(x)=\sum_{i=0}^{n-1} f(W, U^i\sigma)(x).$$
To talk about any notion of activity, we need to assume that the graphon $W$ is \emph{connected}, that is, there is no measurable partition $[0, 1] = A \cup B$ with $\lambda(A)$, $\lambda(B) > 0$ and $W(x, y) = 0$ for almost all $(x, y) \in A \times B$. As we will see in Section \ref{ss:example: activity does not exists}, connectedness itself is not enough: there is a connected graphon $W$ with a reasonably nice chip configuration $\sigma$ such that $\lim_{n  \to \infty}\frac{u_n(W, \sigma)(x)}{n}$ does not exists for any $x \in [0, 1]$. If for a given graphon $W$ and chip configuration $\sigma$ there is a real number $a = a(W, \sigma)$ such that $\lim_{n \to \infty} \frac{u_n(W, \sigma)(x)}{n}$ exists and is equal to $a$ for almost all $x \in [0, 1]$ then we say that the \emph{activity exists} and is equal to $a$. As we will see in Theorem \ref{t:activity_exists_with_FDC}, the activity of any chip configuration exists on a graphon with a lower bound on the degrees.

We can also introduce the activity diagram of a chip configuration $\sigma$ on a graphon $W$ as a straightforward generalization of the graph case: $s(W,\sigma)(y)=a(W,\sigma+y\cdot \deg_W)$. The following claim tells us that activity diagrams are monotone increasing.

\begin{lemma}\label{l:more_chips_fire_more}
	If $\sigma'\geq\sigma$ almost everywhere, then $u_n(W,\sigma')(x)\geq u_n(W,\sigma)(x)$ for any graphon $W$, $n\in \mathbb{N}$ and almost all $x\in[0,1]$.
\end{lemma}
\begin{proof} 
We proceed by induction on $n$. The statement is clear for $n = 0$, since $u_0(W, \sigma')(x) = u_0(W, \sigma)(x)=0$ for all $x \in [0, 1]$. 

Suppose that $u_n(W,\sigma')(x)\geq u_n(W,\sigma)(x)$ for almost all $x\in[0,1]$. Then almost all $x\in[0,1]$ has the properties that $\sigma'(x) \ge \sigma(x)$ and $u_n(W,\sigma')(x)\geq u_n(W,\sigma)(x)$. 
Fix such an $x\in[0,1]$ towards showing that $u_{n+1}(W, \sigma')(x) \ge u_{n + 1}(W, \sigma)(x)$. By induction hypothesis, $u_n(W,\sigma')(x)= u_n(W,\sigma)(x) + k$ for some $k\geq 0$, moreover,
\begin{align*}
U^n\sigma(x)=\sigma(x)-u_n(W,\sigma)(x)\deg_W(x) + \int_0^1 u_n(W,\sigma)(y) W(x,y) \dy\leq \\
\sigma'(x)-(u_n(W,\sigma')(x)-k)\deg_W(x) + \int_0^1 u_n(W,\sigma')(y) W(x,y) \dy = \\
U^n\sigma'(x) + k \deg_W(x).
\end{align*}

It follows that 
\begin{align*}
\begin{split}
    u_{n+1}(W,\sigma)(x) &=u_n(W,\sigma)(x) + \left\lfloor \frac{U^n\sigma(x)}{\deg_W(x)}\right\rfloor \leq u_n(W,\sigma)(x) + \left\lfloor \frac{U^n\sigma'(x)}{\deg_W(x)}\right\rfloor + k \\ &= u_n(W,\sigma')(x) + \left\lfloor \frac{U^n\sigma'(x)}{\deg_W(x)}\right\rfloor = u_{n+1}(W,\sigma')(x).
\end{split}
\end{align*}
This finishes the proof.
\end{proof}

\section{The finite diameter condition}

In this section we formulate a notion for graphons that is an analogue of the diameter of finite graphs. 
We will be able to give a sufficient condition for the existence of the activity of a chip configuration using this notion.

For a measurable set $A\subseteq [0,1]$, we denote by $\Gamma(A)$ the neighborhood of $A$ in $W$, i.e.
$$
\Gamma(A)=\{x\in [0,1]: \exists y\in A \text{ such that }W(y,x)>0\}.
$$

For $\varepsilon>0$, we denote by $\Gamma_\varepsilon(A)$ the set of those neighbors of $A$ that receive at least $\varepsilon$ chips by firing the set $A$ once, that is,
$$
\Gamma_\varepsilon(A)=\left\{x\in [0,1]: \int_{A} W(y,x) \dy\geq \varepsilon\right\}.
$$

We denote $\Gamma^k(A)=\underbrace{\Gamma \circ \dots \circ \Gamma}_k(A)$ and similarly for $\Gamma^k_\varepsilon(A)$.

The following definition plays a key role in our results. 
\begin{defn}[Finite diameter condition] \label{def:finite_diam_cond}
	A graphon $W:[0,1]^2\to [0,1]$ is said to have \emph{finite diameter}, if there is an $N \in \N$ such that for all measurable subset $A \subseteq [0,1]$ with $\lambda(A) > 0$ there exists $\varepsilon > 0$ with $\lambda(A \cup \Gamma_\varepsilon(A) \cup \Gamma^2_\varepsilon(A)\cup \dots \cup \Gamma^N_\varepsilon(A))=1$.
\end{defn}

It is reasonable to call the smallest such $N$ the \emph{diameter} of $W$, but we will not use this notion. There are many equivalent ways to define the finite diameter property. Above we tried to give the most natural definition. The following theorem gives two more equivalent formulations that will play a role in this paper. We also note that we could use $\Gamma'_\varepsilon(A)=A\cup \Gamma_\varepsilon(A)$ and get the same property.

\begin{thm} The following statements are equivalent for a graphon $W$:
  \label{t:finite diam equiv conditions}
  \begin{itemize}
    \item[(i)] $W$ has a finite diameter;
    \item[(ii)] there exist $N \in \N$ and $\varepsilon > 0$ such that for all measurable $A \subseteq [0, 1]$ with $\lambda(A) \ge \frac{1}{2}$, $\lambda(A \cup \Gamma_\varepsilon(A) \cup \Gamma^2_\varepsilon(A)\cup \dots \cup \Gamma^N_\varepsilon(A))=1$;
    \item[(iii)] $W$ is connected and there exists $\delta > 0$ such that $\deg_W(x) \ge \delta$ for almost all $x$.
  \end{itemize}
\end{thm}
Notice that (ii) is different from (i) in that we require the existence of an $\varepsilon$ that is suitable for every ``large'' measurable set.

\begin{proof}
  First we prove that (i) and (ii) both imply (iii). 
  
  If $W$ is not connected and $A \cup B = [0, 1]$ is a partition witnessing this, then by supposing $\lambda(A) \ge 1/2$, we see that for each $N \in \N$ and each $\varepsilon > 0$, $\lambda(A \cup \Gamma_\varepsilon(A) \cup \Gamma^2_\varepsilon(A)\cup \dots \cup \Gamma^N_\varepsilon(A))= \lambda(A) < 1$, contradicting the assumptions of both (i) and (ii). 
  
  If the degrees of $W$ are not bounded from below (so for every $\varepsilon > 0$, $\lambda(\{x : \deg_W(x) < \varepsilon\}) > 0$) then the degrees are not bounded from below on $[0, \frac{1}{2})$ (i.e. for every $\varepsilon > 0$, $\lambda(\{x\in[0,\frac{1}{2}) : \deg_W(x) < \varepsilon\}) > 0$) or on $[\frac{1}{2}, 1]$. Suppose that they are not bounded from below on $[0, \frac{1}{2})$ and let $A = [\frac{1}{2}, 1]$. Then for every $\varepsilon > 0$, the set $B_\varepsilon = \{x \in [0, \frac{1}{2}) : \deg_W(x) < \varepsilon\}$ is of positive measure, so $\lambda(A \cup \Gamma_\varepsilon(A) \cup \Gamma^2_\varepsilon(A)\cup \dots \cup \Gamma^N_\varepsilon(A)) \le \lambda([0, 1] \setminus B_\varepsilon) < 1$ for any $N \in \N$. Thus we have proved the directions (i) $\Rightarrow$ (iii) and (ii) $\Rightarrow$ (iii).
  
  To show (iii) $\Rightarrow$ (i) and (iii) $\Rightarrow$ (ii), we first prove the following. 
  \begin{claim}
    \label{c:compactness}
    If $W$ is connected then for each interval $[a, b] \subset (0, 1)$ there exists $\varepsilon > 0$ so that $\int_A \int_{A^c} W(x, y) \dx\dy \ge \varepsilon$ for all measurable subset $A \subseteq [0, 1]$ with $\lambda(A) \in [a, b]$. 
  \end{claim}
  \begin{proof}
    Suppose towards a contradiction that for some $[a,b]$ there is a sequence of subsets $(A_n)_{n \in \N}$ such that $\lambda(A_n) \in [a, b]$ but $\int_{A_n} \int_{A_n^c} W(x, y) \dx\dy \to 0$. Our goal is to contradict the connectedness of $W$ by coming up with a subset $A \subseteq [0, 1]$ with $\lambda(A) \in [a, b]$ and $\int_{A} \int_{A^c} W(x, y) \dx\dy = 0$. 
    
    Since $\{\mathbf{1}_{A_n} : n \in \N\}$ is a bounded subset of $L^\infty([0, 1])$, and $L^\infty$ is the dual of $L^1$, there is a weak$^*$ convergent subsequence of $(\mathbf{1}_{A_n})_{n\in\mathbb{N}}$ tending to $f \in L^\infty$ by the Banach--Alaoglu theorem. We can suppose that the subsequence is the original one, hence 
    \begin{equation}
      \label{e:weak* limit}
      \int_0^1 g(x) \cdot \mathbf{1}_{A_n}(x) \dx \to \int_0^1 g(x) \cdot f(x) \dx
    \end{equation}
    for all $g \in L^1([0, 1])$. One can easily see that $f(x) \in [0, 1]$ for almost all $x \in [0, 1]$ and $\int_0^1 f(x) \dx \in [a, b]$ by using $g(x) = \mathbf{1}_{\{x : f(x) < 0\}}$, $g(x) = \mathbf{1}_{\{x : f(x) > 1\}}$ and $g(x) \equiv 1$ in \eqref{e:weak* limit}.
    
    It follows from our assumptions on $A_n$ that 
    \begin{align*}
    \int_0^1 \int_0^1 W(x, y) \mathbf{1}_{A_n}(x) \mathbf{1}_{A_n^c}(y) \dx \dy \\= \int_0^1 \int_0^1 W(x, y) \mathbf{1}_{A_n}(x) (1 - \mathbf{1}_{A_n}(y)) \dx \dy \to 0.
    \end{align*}
    Now we show that $\int_0^1 \int_0^1 W(x, y) f(x) (1 - f(y)) \dx \dy = 0$. The function $x \mapsto W(x, y)$ is in $L^1$ for almost all $y \in [0, 1]$, hence, using again \eqref{e:weak* limit}, $$\int_0^1 W(x, y) \mathbf{1}_{A_n}(x) \dx \to \int_0^1 W(x, y) f(x) \dx$$ for almost all $y$. For a fixed $\varepsilon > 0$, let $n_0$ be large enough so that for $$B = \left\{y : \forall n \ge n_0\quad \left| \int_0^1 W(x, y) \mathbf{1}_{A_n}(x) \dx - \int_0^1 W(x, y) f(x) \dx \right| \le \varepsilon\right\},$$ $\lambda(B) \ge 1 - \varepsilon$. Then for each $n \ge n_0$, 
    \begin{align*}
      \bigg|\int_0^1 \int_0^1 W(x, y) \mathbf{1}_{A_n}(x) (1 - \mathbf{1}_{A_n}(y)) \dx \dy \\- \int_0^1 \int_0^1 W(x, y) f(x) (1 - \mathbf{1}_{A_n}(y)) \dx \dy \bigg| \\ \le \int_B \left|\int_0^1 W(x, y) \mathbf{1}_{A_n}(x) \dx - \int_0^1 W(x, y) f(x) \dx \right| (1 - \mathbf{1}_{A_n}(y)) \dy \\ + \int_{B^c} \left|\int_0^1 W(x, y) \mathbf{1}_{A_n}(x) \dx - \int_0^1 W(x, y) f(x) \dx \right| (1 - \mathbf{1}_{A_n}(y)) \dy \le 2\varepsilon,
    \end{align*}
    using that every function here has values in $[0, 1]$ and that $\lambda(B^c) \le \varepsilon$. The function $y \mapsto \int_0^1 W(x, y) f(x) \dx$ is in $L^1$, hence $$\int_0^1 \int_0^1 W(x, y) f(x) (1 - \mathbf{1}_{A_n}(y)) \dx \dy \to \int_0^1 \int_0^1 W(x, y) f(x) (1 - f(y)) \dx \dy.$$ Thus, 
    \begin{align*}
      \limsup_{n \to \infty} \bigg|\int_0^1 \int_0^1 W(x, y) \mathbf{1}_{A_n}(x) (1 - \mathbf{1}_{A_n}(y)) \dx \dy \\ - \int_0^1 \int_0^1 W(x, y) f(x) (1 - f(y)) \dx \dy \bigg| \le 3 \varepsilon
    \end{align*}
    for every $\varepsilon > 0$, meaning that $$\int_0^1 \int_0^1 W(x, y) \mathbf{1}_{A_n}(x) (1 - \mathbf{1}_{A_n}(y)) \dx \dy \to \int_0^1 \int_0^1 W(x, y) f(x) (1 - f(y)) \dx \dy,$$ hence $\int_0^1 \int_0^1 W(x, y) f(x) (1 - f(y)) \dx \dy = 0$.
    
    Now we extract a subset from $f$. Since $\int_0^1 f(x) \dx \in [a, b]$ and $f(x) \in [0, 1]$ for almost all $x$, $\lambda(\{x : f(x) = 1\}) \le b$ and $\lambda(\{x : f(x) > 0\}) \ge a$. Hence, there is a measurable set $A$ with $\{x : f(x) = 1\} \subseteq A \subseteq \{x : f(x) > 0\}$ and $\lambda(A) \in [a, b]$. It is easy to check that if for some $(x, y) \in [0, 1]^2$, $f(x) (1 - f(y)) = 0$, then $\mathbf{1}_A(x) (1 - \mathbf{1}_{A}(y)) = 0$. Therefore $\int_0^1 \int_0^1 W(x, y) f(x) (1 - f(y)) \dx \dy = 0$ implies that $\int_0^1 \int_0^1 W(x, y) \mathbf{1}_A(x)(1 - \mathbf{1}_A(y)) \dx\dy = 0$, hence $A$ witnesses that $W$ is not connected, a contradiction. 
  \end{proof}
  
  Now we prove (iii) $\Rightarrow$ (i) and (iii) $\Rightarrow$ (ii). It is enough to prove the corresponding statements for the operator $\Gamma'_\varepsilon(A) = A \cup \Gamma_\varepsilon(A)$ in place of $\Gamma_\varepsilon(A)$, since one can easily see by induction on $k$ that $(\Gamma'_\varepsilon)^k(A) = A \cup \dots \cup  (\Gamma'_\varepsilon)^k(A) \subseteq A \cup \Gamma_{\varepsilon/N}(A) \cup \dots \cup \Gamma_{\varepsilon/N}^k(A)$ for each $k \le N$. 
  
  The following claim proves (ii) from (iii) and will also be used to prove (i).
  \begin{claim}
    \label{c:N eps for [a, 1]}
    If $W$ is a connected graphon, and $\deg_W(x) \ge \delta$ for almost every $x$ for some $\delta > 0$, then for every $a \in (0, 1]$ there exist $N \in \N$ and $\varepsilon > 0$ such that $\lambda((\Gamma'_\varepsilon)^N(A)) = 1$ for every measurable set $A$ with $\lambda(A) \in [a, 1]$. 
  \end{claim}
  \begin{proof}
    Using the lower bound on the degree,  
    \begin{equation}
      \label{e:from 1-delta/2 to 1}
      \text{$\lambda(\Gamma'_\varepsilon(A)) = 1$ for every measurable $A$ with $\lambda(A) \ge 1 - \frac{\delta}{2}$ and $\varepsilon \le \frac{\delta}{2}$.}
    \end{equation}

      Now let $\varepsilon' > 0$ be given by Claim \ref{c:compactness} for $[a, b] = [a, 1 - \frac{\delta}{2}]$, then for every measurable subset $A$ with $\lambda(A) \in [a, 1 - \frac{\delta}{2}]$, $\int_0^1 \int_0^1W(x, y) \mathbf{1}_A(x) \mathbf{1}_{A^c}(y) \dx\dy \ge \varepsilon'$. Since $\int_0^1 W(x, y) \mathbf{1}_A(x) \dx \le 1$ for almost all $y \in A^c$, 
    $$
      \lambda\left(\left\{y \in A^c : \int_0^1 W(x, y) \mathbf{1}_A(x) \dx \ge \frac{\varepsilon'}{2}\right\}\right) \ge \frac{\varepsilon'}{2}.
    $$
    In other words, $\lambda(\Gamma'_{\varepsilon'/2}(A) \setminus A) \ge \frac{\varepsilon'}{2}$ for every measurable subset $A$ with $\lambda(A) \in [a, 1 - \frac{\delta}{2}]$. Then, $\lambda((\Gamma'_{\varepsilon'/2})^{\lceil 2/\varepsilon' \rceil}(A)) \in [1 - \frac{\delta}{2}, 1]$ for every such $A$. Therefore, also using \eqref{e:from 1-delta/2 to 1}, $N = \lceil \frac{2}{\varepsilon'} \rceil + 1$ and $\varepsilon = \min\{\frac{\varepsilon'}{2}, \frac{\delta}{2}\}$ satisfy the claim. 
  \end{proof}
  
  The proof of (iii) $\Rightarrow$ (ii) is complete using the claim, so now we move on to show (iii) $\Rightarrow$ (i). The extra difficulty comes from sets of small measure. If $A$ is a measurable subset with $0 < \lambda(A) \le \frac{\delta}{3}$ then for almost all $x \in A$, $\int_0^1 W(x, y) \mathbf{1}_{A^c}(y) \dy \ge \frac{2\delta}{3}$, hence $\int_0^1 \int_0^1 W(x, y) \mathbf{1}_A(x) \mathbf{1}_{A^c}(y) \dx\dy \ge \frac{2\delta\lambda(A)}{3}$. Let $$B = \left\{y \in A^c : \int_0^1 W(x, y) \mathbf{1}_A(x) \dx \ge \frac{\delta\lambda(A)}{3}\right\}.$$ Since $\int_0^1 W(x, y) \mathbf{1}_A(x) \dx \le \lambda(A)$ for almost all $y \in A^c$,
  \begin{align*}
    \frac{2\delta\lambda(A)}{3} \le \int_0^1 \int_0^1 W(x, y) \mathbf{1}_A(x) \mathbf{1}_{A^c}(y) \dx\dy \le \lambda(B) \cdot \lambda(A) + \frac{\delta\lambda(A)}{3},
  \end{align*}
  therefore $\lambda(B) \ge \frac{\delta}{3}$, and thus $\lambda(\Gamma'_{\frac{\delta\lambda(A)}{3}}(A)) \ge \frac{\delta}{3}$. 
  
  Now we apply Claim \ref{c:N eps for [a, 1]} with $a = \frac{\delta}{3}$ to get $N'$ and $\varepsilon'$ such that $\lambda((\Gamma'_{\varepsilon'})^{N'}(A)) = 1$ for every $A$ with $\lambda(A) \ge \frac{\delta}{3}$. Then for any subset $A$ of positive measure, $N = N' + 1$ (that is independent of $A$) and $\varepsilon = \min\{\varepsilon', \frac{\delta\lambda(A)}{3}\}$ work, $\lambda((\Gamma'_\varepsilon)^N(A)) = 1$. 
\end{proof}

Let us point out a nice property of graphons with finite diameter. In many respect, unbounded chip configurations are inconvenient. However, for a graphon with finite diameter, any chip configuration becomes bounded after one step of the parallel chip-firing. 

\begin{lemma}
  \label{l:boundedness after one firing}
 If a graphon $W$ has $\mindeg(W) = d >0$, then for any chip configuration $\sigma$ on $W$, $n \geq 1$ and almost all $x \in [0, 1]$, we have $$U^n\sigma(x) \leq \deg_W(x) + \frac{\|\sigma\|_1}{d} \leq 1 + \frac{\|\sigma\|_1}{d}.$$
\end{lemma}
\begin{proof}
Since by Claim \ref{c:L1_norm_stays_the_same_after_firing}, $\|U\sigma
\|_1=\|\sigma\|_1$, it is enough to prove the statement for $n=1$.

$U\sigma(x) = \sigma(x) - \deg_W(x)f(x)  + \int_0^1 f(y) W(x,y)\dy$ where
\begin{equation*}
  f(y) = \left\{ \begin{array}{cl}
    \left\lfloor \frac{\sigma(y)}{\deg_W(y)}\right\rfloor & \text{if }\deg_W(y) > 0 \\
    0 & \text{if }\deg_W(y) = 0.
  \end{array}\right.\ 
\end{equation*}

For an $x$ where $\deg_W(x) > 0$, we have $\sigma(x) - \deg_W(x)f(x)\leq \deg_W(x)$. Also, as $W(x,y)\leq 1$ for each $x,y\in[0,1]$, we have $\int_0^1 f(y) W(x,y)\dy\leq \int_0^1 f(y)\dy\leq \int_0^1 \frac{\sigma(y)}{d} \dy=\frac{1}{d}\|\sigma\|_1$ as $\sigma$ is almost everywhere nonnegative.
\end{proof}

\section{Existence of the activity}
\label{s:existence of activity}

In this section we investigate the existence of the activity of a chip configuration on a graphon. Recall that by definition the activity of a chip configuration $\sigma$ on a graphon $W$  exists and is equal to $a \in \R$ if $\lim_{n \to \infty} \frac{u_n(x)}{n}$ exists and is equal to $a$ for almost all $x \in [0, 1]$. 

First, we construct an example showing that the connectedness of the graphon is not sufficient for the activity to exist. However, we show that if $W$ is connected and $\frac{\sigma(x)}{\deg_W(x)}$ is bounded then $\liminf\frac{u_n(x)}{n}$ is the same for almost every $x \in [0, 1]$, and the same holds for $\limsup\frac{u_n(x)}{n}$. 
In the main result of this section, Theorem \ref{t:activity_exists_with_FDC}, we show that the finite diameter condition implies the existence of the activity for any chip configuration. 

\subsection{An example where the activity does not exist}
\label{ss:example: activity does not exists}
\begin{prop}\label{p:example for activity to not exist}
  There exist a connected graphon $W$ and a bounded chip configuration $\sigma$ on $W$ such that the activity of $\sigma$ does not exist. 
\end{prop}
\begin{proof}
In our construction we will have $\liminf\frac{u_n(x)}{n} = \frac{1}{2}$ and $\limsup \frac{u_n(x)}{n}= 1$ for each $x \in [0, 1]$. 

Let $\bigcup_{m \in \Z} A_m$ be a measurable partition of $[0, 1]$ with $\lambda(A_m) > 0$ for each $m \in \Z$, and let us denote by $m : [0, 1] \to \Z$ the unique function with $x \in A_{m(x)}$ for every $x \in [0, 1]$. Then let $W(x, y) = 1$ if and only if $|m(x) - m(y)| = 1$ and let $W(x, y) = 0$ otherwise. It is easy to see that $W$ is connected.

We say that a set $A_m$ is of type 1 with respect to a chip configuration $\rho$ if $\rho(x) = \lambda(A_{m - 1})$ for each $x \in A_m$. We say that $A_m$ is of type 2 if $\rho(x) = \lambda(A_{m - 1}) + \lambda(A_{m + 1})$ for each $x \in A_m$, and it is of type 3 if $\rho(x) = 2\cdot \lambda(A_{m - 1}) + \lambda(A_{m + 1})$ for each $x \in A_m$. In our example, we will choose the starting configuration $\sigma$ so that for each $n$, each set $A_m$ is of type $i$ for some $i$ with respect to $U^n\sigma$. It is clear that every such chip configuration is bounded.

Now let $Z_1 \cup Z_2 \cup Z_3 = \Z$ be a partition of the integers with the property that 
\begin{equation}
\label{e:zpart}
m \in Z_1 \Leftrightarrow m + 1 \in Z_3.
\end{equation}

We define the chip configuration $\sigma$ in the following way:
\begin{equation*}
  \sigma(x) = \left\{\begin{array}{cl} 
    \lambda(A_{m(x) - 1}) & \text{if $m(x) \in Z_1$}, \\ 
    \lambda(A_{m(x) - 1}) + \lambda(A_{m(x) + 1}) & \text{if $m(x) \in Z_2$},\\ 
    2\cdot \lambda(A_{m(x) - 1}) + \lambda(A_{m(x) + 1}) & \text{if $m(x) \in Z_3$},
  \end{array}
  \right.
\end{equation*}
that is, the type of $A_m$ is $i$ with respect to $\sigma$ if and only if $m \in Z_i$.

Using \eqref{e:zpart}, we have that 
\begin{equation}
  \label{e:spart}
  \text{$A_m$ is of type 1 if and only if $A_{m + 1}$ is of type 3 with respect to $\sigma$.}
\end{equation}

\begin{claim}
  \label{c:types}
  Suppose that a configuration $\rho$ satisfies \eqref{e:spart} and also that each $A_m$ is of some type with respect to $\rho$. Then for each $m \in \Z$, $A_m$ is of type $i$ with respect to $U\rho$ if and only if $A_{m + 1}$ is of type $i$ with respect to $\rho$. 
\end{claim}
\begin{proof}
  We distinguish multiple cases according to the type of $A_m$ with respect to $\rho$. 
  
  If $A_{m}$ is of type 1 with respect to $\rho$ then using \eqref{e:spart}, $A_{m + 1}$ is of type 3 and $A_{m - 1}$ is of type 2 or 3. Hence, starting from $\rho$, the points of $A_{m - 1} \cup A_{m + 1}$ can fire once, but the points of $A_m$ cannot fire at all. Therefore, $U\rho(x) = 2\cdot \lambda(A_{m - 1}) + \lambda(A_{m + 1})$ for each $x \in A_m$, thus $A_m$ is indeed of type 3 with respect to $U\rho$, as is $A_{m + 1}$ with respect to $\rho$. 
  
  If $A_m$ is of type 2 with respect to $\rho$ then again using \eqref{e:spart}, $A_{m - 1}$ is of type 2 or type 3, and $A_{m + 1}$ is of type 1 or type 2. Since in this case the points of $A_{m - 1} \cup A_m$ can fire once, for each $x \in A_m$, $U\rho(x) = \lambda(A_{m - 1})$ if $A_{m + 1}$ is of type 1, and $U\rho(x) = \lambda(A_{m - 1}) + \lambda(A_{m + 1})$ if $A_{m + 1}$ is of type 2. Thus the proof is also complete in this case.
  
  If $A_m$ is of type 3 with respect to $\rho$ then by \eqref{e:spart}, $A_{m - 1}$ is of type 1, and $A_{m + 1}$ is of type 1 or type 2. Hence, for each $x \in A_m$, $U\rho(x) = \lambda(A_{m - 1})$ if $A_{m + 1}$ is of type 1, and $U\rho(x) = \lambda(A_{m - 1}) + \lambda(A_{m + 1})$ if $A_{m + 1}$ is of type 2. Thus the proof of our claim is complete. 
\end{proof}

It is easy to prove by induction on $n$ using Claim \ref{c:types}, that for each $n$, every set $A_m$ is of some type with respect to $U^n\sigma$ and also that $U^n\sigma$ satisfies \eqref{e:spart}. It is also clear that $u_n(x)$ equals the cardinality of the set $(Z_2 \cup Z_3) \cap \{m(x), m(x) + 1, \dots, m(x) + n - 1\}$. The only thing that remains to finish the construction, is to choose the partition $\mathbb{Z}=Z_1 \cup Z_2 \cup Z_3$ such that \eqref{e:zpart} is satisfied and also for every $x$ the liminf of $|(Z_2 \cup Z_3) \cap \{m(x), m(x) + 1, \dots, m(x) + n - 1\}|$ as $n$ tends to infinity is $1/2$ and the limsup is $1$. Let for example $$Z_2 = \{n \in \Z : n \le 1\} \cup \bigcup_{k = 1}^\infty \{n \in \Z : (2k + 1)! \le n < (2k + 2)!\},$$
where $k!$ denotes the factorial of $k$. We can add the remaining integers alternatingly to $Z_1$ and $Z_3$ satisfying \eqref{e:zpart}. It is straightforward to check that this construction satisfies the above requirements.
\end{proof}

\iftrue
\subsection{About the $\liminf$ and $\limsup$}
\label{s:liminf limsup}

In this section we prove the following theorem.

\begin{thm}
 \label{t:liminf, limsup}
 If $W$ is a connected graphon and $\sigma$ is a chip configuration on $W$ such that $\frac{\sigma(x)}{\deg_W(x)} < K$ almost everywhere, then there are real numbers $\underline{u}$, $\overline{u} \in [0, 1]$ such that $\liminf \frac{u_n(x)}{n} = \underline{u}$ and $\limsup \frac{u_n(x)}{n} = \overline{u}$ for almost all $x \in [0, 1]$.
\end{thm}

We conjecture that this statement holds more generally, for any chip configuration.

\begin{conject}
  If $W$ is a connected graphon and $\sigma$ is a chip configuration on $W$, then there are real numbers $\underline{u}$, $\overline{u} \in [0, 1]$ such that $\liminf \frac{u_n(x)}{n} = \underline{u}$ and $\limsup \frac{u_n(x)}{n} = \overline{u}$ for almost all $x \in [0, 1]$.
\end{conject}

The idea of the proof is the following. The function $y \mapsto \frac{W(x, y)}{\deg_W(x)}$ can be interpreted as a density describing the neighborhood of $x$. Even though these densities can be quite different for different points, if we start a Markov chain at each point, the transition probabilities of this Markov chain will be close to each other after a sufficiently large number of steps. We can approximate the amount of chips received by the points using these probabilities, showing that the $\liminf$ and $\limsup$ of $\frac{u_n(x)}{n}$ do not depend on $x$. 

We start with collecting the notions regarding Markov chains that we will need. Let $(X, \mathcal{A})$ be a measurable space and $P : X \times \mathcal{A} \to [0, 1]$ denote the transition probabilities of a Markov chain, that is, $P(x, \cdot)$ is a probability distribution for each $x \in X$ and $P(\cdot, A)$ is measurable for each $A \in \mathcal{A}$. The higher-order transition probabilities are defined by $P^1(x, A) = P(x, A)$ and 
$$
 P^{n + 1}(x, A) = \int_{X} P^n(x, \dy) P(y, A).
$$

The Markov chain defined by $P$ is said to be \emph{irreducible} if there exists a non-zero $\sigma$-finite measure $\phi$ on $X$ such that for every $A \in \mathcal{A}$ with $\phi(A) > 0$ and every $x \in X$ there exists $n \in \N$ with $P^n(x, A) > 0$.

The probability distribution $\pi$ is called \emph{stationary distribution} if
$$
 \pi(A) = \int_{X} P(x, A) d\pi(x)
$$
for every $A \in \mathcal{A}$.
A Markov chain with a stationary distribution $\pi$ is \emph{aperiodic} if there do not exist disjoint, measurable subsets $A_0, \dots, A_{d - 1} \subseteq X$ with $d \ge 2$ such that for all $0\leq i < d$ and all $x \in A_i$, $P(x, A_{(i + 1) \pmod{d}}) = 1$, and $\pi(A_0) > 0$.

We will use the following result, appearing in this form in \cite[Theorem 4]{Markov chain}, in which the norm of a signed measure is defined as the \emph{total variation norm}, that is, 
$\|\nu\| = \sup_{A \in \mathcal{A}} |\nu(A)|$. 

\begin{thm}
 \label{t:Markov chain}
 If a Markov chain on a state space with countably generated $\sigma$-algebra is irreducible and aperiodic, and has a stationary distribution $\pi$, then for $\pi$-a.e.~$x \in X$,
$$
 \lim_{n \to \infty} \|P^n(x,\cdot) - \pi(\cdot)\| = 0.
$$
\end{thm}

Now we are ready to prove the main result in this subsection.

\begin{proof}[Proof of Theorem \ref{t:liminf, limsup}]
 We first claim that we can suppose that $\deg_W(x) > 0$ for every $x \in [0, 1]$. In order to show this, let $A = \{ x \in [0, 1] : \deg_W(x) = 0\}$, and define $W'$ by $W'(x, y) = 1$ if $x \in A$ or $y \in A$, and $W'(x, y) = W(x, y)$ otherwise. It is easy to check that $W'$ is a graphon, and $W(x, y) = W'(x, y)$ for almost all $(x, y) \in [0, 1]^2$, since the connectedness of $W$ implies $\lambda(A) = 0$. Hence, $W'$ is also connected and $\deg_{W'}(x) > 0$ for all $x$. One can also show by induction on $n$, that if $\deg_W(x) > 0$ then $u_n(W, \sigma)(x) = u_n(W', \sigma)(x)$, since in this case $W'(x, y) = W(x, y)$ for almost all $y \in [0, 1]$. Hence if we show the statement of the theorem for $W'$ then we are also done for $W$. Thus we can indeed suppose that $\deg_W(x) > 0$ for all $x \in [0, 1]$. We also assume that $W$ is Borel measurable; we can do so, since for any (Lebesgue measurable) graphon $W$ there is a Borel measurable one $W'$ such that $W = W'$ almost everywhere, hence, using that $\deg_{W'}(x) = \deg_W(x)$ almost everywhere, it follows that $u_n(W', \sigma)(x) = u_n(W, \sigma)(x)$ almost everywhere. 
  
 Define a Markov chain using $W$ by the transition probabilities 
 \begin{equation*}
   P(x, A) = \frac{\int_A W(x, y) \dy}{\deg_W(x)}.
 \end{equation*}
 Here, and everywhere else where we do not specify the measure, we integrate with respect to the Lebesgue measure.
  
 To check that $P : [0, 1] \times \mathcal{B}([0, 1]) \to [0, 1]$, where $\mathcal{B}([0, 1])$ is the Borel $\sigma$-algebra, indeed defines a Markov chain, one can use e.g. \cite[Exercise 17.36]{K} to show first that $x \mapsto \deg_W(x)$ is Borel, hence $(x, y) \mapsto \frac{W(x, y)}{\deg_W(x)}$ is also Borel, and then use the same exercise again to show that for each $A \in \mathcal{B}([0, 1])$, $x \mapsto P(x, A)$ is also Borel.
  
 Notice that each distribution $P(x, \cdot)$ is absolutely continuous with respect to the Lebesgue measure, since 
 \begin{equation*}
   p(x, y) = \frac{W(x, y)}{\deg_W(x)}
 \end{equation*}
 is a density function. One can also show by induction on $n$ that $P^n(x, \cdot)$ has a density function $p^n(x, \cdot)$, which can be defined inductively by $p^1(x, y) = p(x, y)$ and 
 $$
   p^{n + 1}(x, y) = \int_0^1 p^n(x, z)p(z, y) \dz = \int_0^1 p^n(x, z) \frac{W(z, y)}{\deg_W(z)} \dz.
 $$
 We will also use that for $k \le n$, 
 \begin{equation}
   \label{e:p^n = int p^k p^n-k}
   p^n(x, y) = \int_0^1 p^k(x, z) p^{n - k}(z , y) \dz.
 \end{equation}
  
 Let us define a probability distribution on $[0, 1]$ by $\pi(A) = \frac{\int_A \deg_W(x) \dx}{\int_0^1 \deg_W(x) \dx}$ with density function $d_\pi(x) = \frac{\deg_W(x)}{\int_0^1 \deg_W(y) \dy}$. 
 \begin{claim}
   \label{c:irreducible, stationary}
   The Markov chain determined by $P$ is irreducible, and has $\pi$ as a stationary distribution.
 \end{claim}
 \begin{proof}
   The following calculation, using Fubini's theorem for non-negative functions shows that $\pi$ is indeed a stationary distribution: 
   \begin{equation*}
   \begin{split}
     \int_0^1 P(x, A) d\pi(x) &= \int_0^1 P(x, A) d_\pi(x) \dx = 
     \int_0^1 \int_{A} p(x, y) \dy \cdot d_\pi(x) \dx \\
     &= \int_0^1 \int_{A} \frac{W(x, y)}{\deg_W(x)} \dy \cdot \frac{\deg_W(x)}{\int_0^1 \deg_W(z) \dz} \dx \\
     &= \int_0^1 \int_{A} W(x, y) \dy \dx \cdot \frac{1}{\int_0^1 \deg_W(z) \dz} \\
     &= \frac{\int_{A} \deg_W(y) \dy}{\int_0^1 \deg_W(z) \dz} = \pi(A).
   \end{split}
   \end{equation*}
    
 	To check irreducibility, we use $\phi = \lambda$, the Lebesgue measure. Let $x \in [0, 1]$ be arbitrary, and let $A_n = \{y \in [0, 1] : p^n(x, y) > 0 \}$. We claim that it is enough to show that $\lambda\left(\bigcup_n A_n\right) = 1$. Indeed, if this is the case and $A \in \mathcal{B}([0, 1])$ with $\lambda(A) > 0$, then $\lambda(A \cap A_n) > 0$ for some $n$, hence $P^n(x, A) = \int_A p^n(x, y) \dy \ge \int_{A\cap A_n} p^n(x, y) \dy > 0$, since $p^n(x, \cdot)$ is positive on $A \cap A_n$ and $\lambda(A \cap A_n) > 0$. 
    
   Now suppose towards a contradiction that $\lambda\left(\bigcup_n A_n\right) < 1$ and let $B = [0, 1] \setminus \bigcup_n A_n$. Then $\lambda(B) > 0$, and also $\lambda(B) < 1$, since  for example $1 = \int_0^1 p^n(x, y) \dy = \int_{A_n} p^n(x, y) \dy$, showing that each $A_n$ is of positive measure. Using the fact that $W$ is connected, $\int_{B^c} \int_{B} W(x, y) \dy \dx > 0$, thus there exists some $n$ such that $\int_{A_n} \int_B W(x, y) \dy \dx > 0$. It follows that for $A_n' = \{ x \in A_n : \int_B W(x, y) \dy > 0\}$, $\lambda(A_n') > 0$. Now we show that $P^{n + 1}(x, B) > 0$ to get a contradiction and complete the proof, again using Fubini's theorem. 
   \begin{equation*}
   \begin{split}
     P^{n + 1}(x, B) &= \int_B p^{n + 1}(x, y)\dy = \int_B \int_0^1 p^n(x, z) p(z, y) \dz \dy \\ 
     &\ge \int_B \int_{A_n'}p^n(x, z) p(z, y) \dz \dy = \int_{A_n'} p^n(x, z) \int_B \frac{W(z, y)}{\deg_W(z)} \dy\dz > 0, 
   \end{split}
   \end{equation*}
   since $p^n(x, z) > 0$ for all $z \in A_n' \subseteq A_n$, $\int_B \frac{W(z, y)}{\deg_W(z)} \dy \ge \int_B W(z, y) \dy > 0$ for each $z \in A_n'$, and $\lambda(A_n') > 0$. 
 \end{proof}
  
 To be able to use Theorem \ref{t:Markov chain}, we would need to prove that our Markov chain is aperiodic. Unfortunately this is not the case if $W$ is \emph{bipartite}, that is, there is a measurable partition $[0, 1] = A \cup B$ such that $\lambda(A)$, $\lambda(B) > 0$ and $W(x, x') = 0$ for almost all $(x, x') \in A^2$ and also $W(y, y') = 0$ for almost all $(y, y') \in B^2$. For our purposes, another formulation of bipartiteness will be useful. For a bipartite graphon $W$ we call the decomposition $[0, 1] \supseteq F = X \cup Y$ into disjoint subsets a \emph{canonical decomposition}, if $\lambda(X)$, $\lambda(Y) > 0$, $\lambda(F) = 1$, for all $x \in X$ and for almost all $x' \in X$, $W(x, x') = 0$, and also for all $y \in Y$ and for almost all $y' \in Y$, $W(y, y') = 0$. Note that every bipartite graphon has a canonical decomposition. 
 
 For such a decomposition we denote by $\pi_X$ the distribution on $X$ defined by $\pi_X(A) = \frac{\int_A \deg_W(x) \dx}{\int_X \deg_W(x) \dx}$, and similarly $\pi_Y$ is a distribution on $Y$ defined by $\pi_Y(B) = \frac{\int_B \deg_W(y) \dy}{\int_Y \deg_W(y) \dy}$. We denote the corresponding density functions by $d_{\pi_X}$ and $d_{\pi_Y}$.
 \begin{claim}
   \label{c:irreducible, stationary for bipart}
   If $W$ is bipartite with canonical decomposition $X \cup Y$, then $P^2(x, \cdot)$ for $x \in X$ are transition probabilities for an irreducible Markov chain on $X$ with stationary distribution $\pi_X$. The analogous statement holds for $Y$ as well. 
 \end{claim}
 \begin{proof}
   Using the fact that $P(x, \cdot)$ is absolutely continuous with respect to the Lebesgue measure, and the properties of the canonical decomposition, one can easily show by induction on $n$ that 
   \begin{equation}
     \label{e:bipartite}
     P^{2n}(x, X) = 1 \text{ and } P^{2n + 1}(x, Y) = 1 \text{ for every $x \in X$.}
   \end{equation}
   This implies that $P^{2}(x, \cdot)$ is a probability distribution on $X$ for each $x \in X$. The measurability of $P^{2}(\cdot, A)$ can be shown as before for $P(\cdot, A)$, showing that $P^2$ indeed defines a Markov chain. 
    
   Claim \ref{c:irreducible, stationary} and \eqref{e:bipartite} shows that $P^2$ is irreducible on $X$, and a similar computation as in the proof of Claim \ref{c:irreducible, stationary} shows that $\pi_X$ is a stationary distribution.  
 \end{proof}
 \begin{claim}
   \label{c:almost everywhere convergence}
   If $W$ is not bipartite then $\|p^n(x, \cdot) - d_\pi\|_1 \to 0$ for almost every $x \in [0, 1]$.
   If $W$ is bipartite with canonical decomposition $X \cup Y$, then $\|p^{2n}(x, \cdot) - d_{\pi_X}\|_1 \to 0$ for almost all $x \in X$ and $\|P^{2n}(y, \cdot) - d_{\pi_Y}\|_1 \to 0$ for almost all $y \in Y$. 
 \end{claim}
 \begin{proof}
   It is enough to prove that if $W$ is not bipartite then $\|P^n(x, \cdot) - \pi\| \to 0$ for almost every $x \in [0, 1]$, and analogously for the bipartite case, since if $\nu$ is a signed measure with density function $d_\nu$, then $\|d_\nu\|_1 = \int |d_\nu(x)|\dx = \big|\int_{d_\nu > 0} d_\nu(x) \dx\big| + \big|\int_{d_\nu < 0} d_\nu(x) \dx\big| \le 2\cdot \|\nu\|$. 
    
   It is also clear that $\pi$ and $\lambda$ are mutually absolutely continuous, since $\pi$ has a density function which is everywhere positive. Hence, a statement holds $\pi$-a.e. if and only if it holds $\lambda$-a.e. In the following discussion, where we can choose between the two, we always use the Lebesgue measure, as in the statement of this claim.
    
   To prove the first assertion we want to apply Theorem \ref{t:Markov chain} for $P$. Using Claim \ref{c:irreducible, stationary}, it remains to show that $P$ is aperiodic. Suppose towards a contradiction that the Markov chain is not aperiodic, hence there exist measurable subsets $A_0, \dots, A_{d - 1} \subseteq [0, 1]$ with $d \ge 2$ such that for all $i < d$ and all $x \in A_i$, $P(x, A_{(i + 1) \pmod{d}}) = 1$, and $\lambda(A_0) > 0$. It follows that $\lambda(A_i) > 0$ for each $i$, since if we suppose this for some fixed $i \le d - 2$ then $1 = P(x, A_{i + 1}) = \int_{A_{i + 1}} p(x, y) \dy$ for all $x \in A_i$, showing $\lambda(A_{i + 1}) > 0$. Let $A = \bigcup_{i < d} A_i$. Then $P(x, A) = 1$ for each $x \in A$, hence $W(x, y) = 0$ for almost all $(x, y) \in A \times A^c$. Then $\lambda(A) = 1$ follows from the connectedness of $W$. 
    
   Now we use the following lemma to show that $d \le 2$. 
   \begin{lemma}
     \label{l:P^2(x, A) > 0 for a.e. x in A}
     For a connected graphon $W$ and a set $A \subseteq [0, 1]$ with $\lambda(A) > 0$, $P^2(x, A) > 0$ for almost all $x \in A$.
   \end{lemma}
   \begin{proof}
     Let $B = \{x \in A : P^2(x, A) = 0\}$. Then, of course $P^2(x, B) = 0$ for all $x \in B$. Suppose towards a contradiction that $\lambda(B) > 0$. Then 
     \begin{equation*}
     \begin{split}
       0 &= \int_B P^2(x, B) \dx = \int_B \int_B p^2(x, z) \dz \dx = \int_B \int_B \int_0^1 p(x, y) p(y, z) \dy \dz \dx \\
       &= \int_0^1 \int_B p(x, y)\dx \int_B p(y, z) \dz \dy. 
     \end{split}
     \end{equation*}
     Since $p(x, y) = 0$ if and only if $p(y, x) = 0$, we have $\int_B p(x, y) \dx = 0$ if and only if $\int_B p(y, z) \dz = 0$. Therefore $p(x, y) = 0$ for almost all $(x, y) \in B \times [0, 1]$, contradicting the fact that $W$ is connected.
   \end{proof}
    
   We can use the lemma for $A = A_0$ to get that $P^2(x, A_0) > 0$ for almost all $x \in A_0$, hence it is not possible to have $d \ge 3$ with $P^2(x, A_2) = 1$ for all $x \in A_0$ and $A_0 \cap A_2 = \emptyset$. Therefore $d \le 2$, and $d = 2$ implies that $W$ is bipartite. It follows that if $W$ is not bipartite then $P$ is aperiodic. The first assertion then follows from Theorem \ref{t:Markov chain}. 
    
   Now suppose that $W$ is bipartite with canonical decomposition $X \cup Y$. To use Theorem \ref{t:Markov chain} to finish the proof, after applying Claim \ref{c:irreducible, stationary for bipart}, it remains to show that $P^2$ is aperiodic on $X$. If this was not the case, there would exist disjoint sets of positive measure $A_0, A_1 \subseteq X$ such that $P^2(x, A_1) = 1$ for all $x \in A_0$. Using Lemma \ref{l:P^2(x, A) > 0 for a.e. x in A} we obtain that $P^2(x, A_0) > 0$ for almost all $x \in A_0$, which contradicts the existence of such sets. 
 \end{proof}
  
 We are now ready to finish the proof of Theorem \ref{t:liminf, limsup}. We calculate lower and upper estimates for $u_n(x_0)$ for an arbitrary $x_0 \in [0, 1]$, after proving the following claim. 
 
 \begin{claim}
   \label{c:u_n(x) <= (K - 1)n}
   If $W$ is a connected graphon and $\sigma$ is a chip configuration on $W$ such that $\frac{\sigma(x)}{\deg_W(x)} < K$ almost everywhere, then $u_n(x) \le (K - 1)n$ for almost all $x \in [0, 1]$.
 \end{claim}
 \begin{proof}
   We show by induction on $n$ that $\frac{U^n\sigma(x)}{\deg_W(x)} < K$. From this statement, the claim easily follows. 
   The statement for $n = 0$ is an assumption of the theorem. Suppose now that it holds for some $n \in \N$ towards showing it for $n + 1$. Clearly, $$U^{n + 1}\sigma(x) < \deg_W(x) + \int (K - 1)W(x, y) \dy < K \deg_W(x),$$ finishing the proof.
 \end{proof}
 
 Since $\int_0^1 W(x, y) u_{n - 1}(y) \dy$ is the amount of mass received by $x$ during the first $n - 1$ steps, $u_n(x) \ge \frac{\int_0^1 W(x, y) u_{n - 1}(y) \dy}{\deg_W(x)} - 1$. Now let $x_0 \in [0, 1]$ be arbitrary and $k \le n$, then 
 \begin{align*}
   u_n(x_0) &\ge \frac{\int_0^1 W(x_0, x_1) u_{n - 1}(x_1) d x_1}{\deg_W(x_0)} - 1 
   =  \int_0^1 p(x_0, x_1) u_{n - 1}(x_1) d x_1 - 1 \\
   &\ge \int_0^1 p(x_0, x_1) \left(\frac{\int_0^1 W(x_1, x_2) u_{n - 2}(x_2) d x_2}{\deg_W(x_1)}  - 1\right) d x_1 - 1 \\
   &= \int_0^1 \int_0^1 p(x_0, x_1) p(x_1, x_2) u_{n - 2}(x_2) d x_1 d x_2 -  \int_0^1 p(x_0, x_1) d x_1 - 1 \\
   &= \int_0^1 p^2(x_0, x_2) u_{n - 2}(x_2) d x_2 - 2 \\
  &\ge  \cdots \\
   &\ge  \int_0^1 p^k(x_0, x_k) u_{n - k}(x_k) d x_k - k. 
 \end{align*}
  
 For the upper estimate, we use that $u_n(x) \le \frac{\sigma(x) + \int_0^1 W(x, y) u_{n - 1}(y) \dy}{\deg_W(x)}$ for any $x$, hence
  
 \allowdisplaybreaks
 \begin{align*}
   u_n(x_0) &\le  
   \frac{\sigma(x_0) + \int_0^1 W(x_0, x_1) u_{n - 1}(x_1) \dx_1}{\deg_W(x_0)} \\
   &=  \frac{\sigma(x_0)}{\deg_W(x_0)} + \int_0^1 p(x_0, x_1) u_{n - 1}(x_1) d x_1 \\
   & \le \frac{\sigma(x_0)}{\deg_W(x_0)} + 
   \int_0^1 p(x_0, x_1) \frac{ \sigma(x_1) + \int_0^1 W(x_1, x_2) u_{n - 2}(x_2) \dx_2}{\deg_W(x_1)} d x_1 \\
   &= \frac{\sigma(x_0)}{\deg_W(x_0)} + \int_0^1 p(x_0, x_1)\frac{\sigma(x_1)}{\deg_W(x_1)}\dx_1 \\
   & \;\;\;\;\; + \int_0^1 p(x_0, x_1) \frac{\int_0^1W(x_1, x_2)u_{n - 2}(x_2) d x_2}{\deg_W(x_1)} d x_1 \\
   &= \frac{\sigma(x_0)}{\deg_W(x_0)} + \int_0^1 p(x_0, x_1)\frac{\sigma(x_1)}{\deg_W(x_1)}\dx_1 \\
      & \;\;\;\;\; + \int_0^1 p^2(x_0, x_2) u_{n - 2}(x_2) \dx_2 \\
   &\le \cdots \\
   &\le \frac{\sigma(x_0)}{\deg_W(x_0)} + \sum_{i = 1}^{k - 1} \left(\int_0^1 p^i(x_0, x_i)\frac{\sigma(x_i)}{\deg_W(x_i)}\dx_i\right) \\
   & \;\;\;\;\; + \int_0^1 p^k(x_0, x_k) u_{n - k}(x_k) \dx_k
 \end{align*}
  
 Now first suppose that $W$ is not bipartite and let $M \subseteq [0, 1]$ be the set of points $x$ such that $\|p^k(x, \cdot) - d_\pi\|_1 \to 0$. Using Claim \ref{c:almost everywhere convergence}, $\lambda(M) = 1$. We now show that the conclusion of the theorem holds for points in $M$, that is, for any $x, x' \in M$, $\liminf \frac{u_n(x)}{n} = \liminf \frac{u_n(x')}{n}$ and $\limsup \frac{u_n(x)}{n} = \limsup \frac{u_n(x')}{n}$. 
 
 Let $x, x' \in M$ be arbitrary, and for a fixed $\varepsilon > 0$ choose $k \in \N$ so that $\|p^k(x, \cdot) - d_{\pi}\|_1 \le \varepsilon$ and $\|p^k(x', \cdot) - d_{\pi}\|_1 \le \varepsilon$. 
 Then, using Claim \ref{c:u_n(x) <= (K - 1)n} and that $\frac{\sigma(x)}{\deg_W(x)} < K$ for almost all $x \in [0, 1]$, for $n \ge k$,
 \begin{align*}
   u_n(x) - u_n(x') &\le \frac{\sigma(x)}{\deg_W(x)} + \sum_{i = 1}^{k - 1} \left(\int_0^1 p^i(x, y)\frac{\sigma(y)}{\deg_W(y)}\dy\right) \\
   & \;\;\;\;\; + \int_0^1 \left(p^k(x, y) - p^k(x', y)\right) u_{n - k}(y) \dy + k \\
   &\le \frac{\sigma(x)}{\deg_W(x)} + \sum_{i = 1}^{k - 1} \left(\int_0^1 p^i(x, y)K \dy\right) \\
   & \;\;\;\;\; + \int_0^1 n(K - 1)\left|p^k(x, y) - p^k(x', y)\right| \dy + k \\
   &\le \frac{\sigma(x)}{\deg_W(x)} + kK + 2n(K - 1)\varepsilon + k. 
 \end{align*}
 One can similarly calculate a lower estimate for $u_n(x) - u_n(x')$, hence as $n$ tends to infinity, we get that $\limsup \big|\frac{u_n(x)}{n} - \frac{u_n(x')}{n}\big| \le 2(K - 1) \varepsilon$ for all $\varepsilon > 0$, hence $\lim\big|\frac{u_n(x)}{n} - \frac{u_n(x')}{n}\big| = 0$, thus $\liminf \frac{u_n(x)}{n} = \liminf \frac{u_n(x')}{n}$ and $\limsup \frac{u_n(x)}{n} = \limsup \frac{u_n(x')}{n}$ for all $x, x' \in M$. Thus the proof of the theorem is complete in case $W$ is not bipartite.
  
 Now suppose that $W$ is bipartite with canonical decomposition $X \cup Y$. Let $M$ be the union of the set of points $x \in X$ with $\|p^{2n}(x, \cdot) - d_{\pi_X}\|_1 \to 0$ and the set of points $y \in Y$ with $\|p^{2n}(y, \cdot) - d_{\pi_Y}\|_1 \to 0$. Using Claim \ref{c:almost everywhere convergence}, $\lambda(M) = 1$. A similar argument to the above one shows that if $x, x' \in X \cap M$ then $\lim\big|\frac{u_n(x)}{n} - \frac{u_n(x')}{n}\big| = 0$ and the same conclusion holds for points $y, y' \in Y \cap M$. It remains to show the same for a pair $(x, y)$ with $x \in X \cap M$ and $y \in Y \cap M$. 
 
 In the following, the density functions $d_{\pi_X}$ and $d_{\pi_Y}$ of $\pi_X$ and $\pi_Y$ are understood to be defined on $[0, 1]$, with $d_{\pi_X}$ vanishing outside $X$ and $d_{\pi_Y}$ vanishing outside $Y$. We claim that for a fixed $\varepsilon$ we can choose $k \in \N$ so that $\|p^k(x, \cdot) - d_{\pi_X}\|_1 \le \varepsilon$ and $\|p^{k + 1}(y, \cdot) - d_{\pi_X}\|_1 \le \varepsilon$. To show this, first note that for all $k \in \N$, 
 \begin{align*}
   \left\|p^{k + 1}(y, \cdot)- d_{\pi_X}\right\|_1 &= \int \left|\int p(y, u)p^k(u, z) \du - d_{\pi_X}(z)\right| \dz \\
   &= \int \left| \int p(y, u)\left(p^k(u, z) - d_{\pi_X}(z)\right) \du \right|\dz \\
   &= \int p(y, u) \left\|p^k(u, .) - d_{\pi_X}\right\|_1 \du.
 \end{align*}
 
 Let $k$ be large enough so that $\|p^k(x, \cdot) - d_{\pi_X}\|_1 \le \varepsilon$ and for that set $H = \{u \in X : \|p^k(u, \cdot) -d_{\pi_X}\|_1 \ge \frac{\varepsilon}{2}\}$, $\lambda(H) \le \frac{\varepsilon \deg_W(y)}{4}$. Then, using that $p(y, u) = 0$ for almost all $u \in Y$, and that $p(y, u) \le \frac{1}{\deg_W(y)}$ and $\|p^k(u, \cdot) - d_{\pi_X}\|_1 \le 2$ for all $u \in [0, 1]$,
  \begin{align*}
    \left\|p^{k + 1}(y, \cdot)- d_{\pi_X}\right\|_1 
    &= \int p(y, u) \left\|p^k(u, .) - d_{\pi_X}\right\|_1 \du \\
    &\le \int_H \frac{2}{\deg_W(y)}  \du + \int_{X \setminus H} p(y, u) \frac{\varepsilon}{2} \du \le \frac{\varepsilon}{2} + \frac{\varepsilon}{2} = \varepsilon,
  \end{align*}
  showing our claim. 
 
 Then using Claim \ref{c:u_n(x) <= (K - 1)n} again, for $n \ge k$,
  
 \begin{align*}
   u_n(x) - u_{n + 1}(y) &\le \frac{\sigma(x)}{\deg_W(x)} + \sum_{i = 1}^{k - 1} \left(\int_0^1 p^i(x, z)\frac{\sigma(z)}{\deg_W(z)}\dz\right) \\
   & \;\;\;\;\; + \int_0^1 \left(p^k(x, z) - p^{k + 1}(y, z)\right) u_{n - k}(z) \dz + (k + 1)\\
   &\le \frac{\sigma(x)}{\deg_W(x)} + kK + 2n(K - 1)\varepsilon + (k + 1), 
 \end{align*}
 with a similar calculation showing the opposite direction, proving together that $\lim\big|\frac{u_n(x)}{n} - \frac{u_{n + 1}(y)}{n}\big| = 0$.  
This implies that $\liminf \frac{u_n(z)}{n}$ and $\limsup\frac{u_n(z)}{n}$ is the same for almost all $z \in [0, 1]$ even if $W$ is bipartite. 
Therefore the proof of the theorem is complete.
\end{proof}

\fi

\subsection{Existence of the activity for graphons with the finite diameter condition}
\label{ss:existence of activity}

In this section we give a sufficient condition for the existence of the activity. 

\begin{thm}\label{t:activity_exists_with_FDC}
	If the finite diameter condition holds for a graphon $W$, then the activity exists for any chip configuration.  
\end{thm}

We believe that having finite diameter is not necessary for the existence of the activity of every chip configuration.

\begin{prob}
Give a necessary and sufficient condition for a graphon $W$ such that the activity $a(W,\sigma)$ exists for each $\sigma$.
\end{prob}

We start the proof of Theorem \ref{t:activity_exists_with_FDC} by investigating the properties of the following two quantities: for a graphon $W$, a chip configuration $\sigma$, and each $n \in \N$ let
\begin{align*}
  m_n &= m_n(W,\sigma) =\inf\{k: \lambda(\{x : u_n(x) = k\})>0\}, \\
  M_n &= M_n(W,\sigma) =\sup\{k: \lambda(\{x : u_n(x) = k\})>0\}. \\
\end{align*}
It is easy to see that $m_n(W, \sigma)$ is finite, however, $M_n(W, \sigma)$ could be infinite. 

\begin{lemma}\label{l:m_n superadd}
	$m_n$ is superadditive, that is, $m_{n+k}\geq m_n + m_k$.
\end{lemma}

\begin{proof}
	We claim that it is enough to prove that 
	\begin{equation}\label{eq:ind_cond_subadd}
	\text{for almost all }x\in[0,1],\ u_{n+k}(x)\geq m_n + u_k(x).
	\end{equation}
    Indeed, suppose that \eqref{eq:ind_cond_subadd} holds and let $A = \{ x \in [0, 1] : u_{n+k}(x)=m_{n+k}$. Note that $\lambda(A) > 0$ by the definition of $m_{n + k}$. For almost all $x \in A$, we have $m_{n+k} = u_{n + k}(x) \geq m_n + u_k(x)$. Since $\lambda(A)>0$, and $u_k(x) < m_k$ can only hold on a set of measure zero, we conclude that there exists $x \in A$ such that $u_k(x) \ge m_k$ and $m_{n+k} \geq m_n + u_k(x)$ both hold. Hence $m_{n+k} \geq m_n + m_k$. 

  To prove (\ref{eq:ind_cond_subadd}), we proceed by induction on $k$. For $k=0$, the statement is trivial. Suppose that the statement holds for $k$, i.e. $u_{n+k}(y)\geq m_n + u_k(y)$ for almost all $y\in[0,1]$. We prove it for $k+1$.
   
  Fix an arbitrary $x\in[0,1]$ with $\deg_W(x) > 0$, and suppose that $u_{n+k}(x)= m_n + u_k(x) +  a$ for some nonnegative integer $a$. If starting from $\sigma$ we fire each vertex $y$ exactly $m_n + u_k(y)$ times, we get to $U^k\sigma$ as firing each vertex $m_n$ times does not change the chip configuration. Now 
  
  \begin{align*}
  U^{n+k}\sigma(x)=\sigma(x)-u_{n+k}(x)\deg_W(x) + \int_0^1 u_{n+k}(y) W(x,y) \dy\geq \\\sigma(x)-(m_n + u_k(x) + a)\deg_W(x) + \int_0^1 (m_n + u_k(y)) W(x,y) \dy = \\
  U^k \sigma(x) -a\cdot\deg_W(x).
  \end{align*}
  
  Hence
  \begin{align*}
  u_{n+k+1}(x)=u_{n+k}(x) + \left\lfloor \frac{U^{n+k} \sigma(x)}{\deg_W(x)}\right\rfloor \geq \\ m_n + u_k(x) + a + \left\lfloor \frac{U^{k} \sigma(x)-a\cdot\deg_W(x)}{\deg_W(x)}\right\rfloor = m_n + u_{k+1}(x).
 \end{align*}
\end{proof}

\begin{lemma}\label{l:M_n subadd}
	$M_n$ is subadditive, that is, $M_{n+k} \leq M_n + M_k$.
\end{lemma}
\begin{proof}
 If $M_n$ is infinite, then we are ready.
 If $M_n$ is finite, then the statement can be proved analogously to Lemma \ref{l:m_n superadd}.
\end{proof} 

We recall Fekete's lemma \cite{Fekete}, that states that for a superadditive sequence $a_n$, $\lim_{n \to\infty} \frac{a_n}{n}$ exists and equals to $\sup_n \frac{a_n}{n}$, and for a subadditive sequence $b_n$, $\lim_{n \to\infty} \frac{b_n}{n}$ exists and equals to $\inf_n \frac{b_n}{n}$. Hence we have the following.

\begin{prop}
    \label{p:lim m_n/n exists}
	The limit $\lim_{n \to\infty} \frac{m_n}{n}$ exists and is equal to $\sup_{n} \frac{m_n}{n}$. The limit $\lim_{n \to\infty} \frac{M_n}{n}$ exists and is equal to $\inf_{n} \frac{M_n}{n}$. In particular, if the activity of $(W, \sigma)$ exists then $$\frac{m_k}{k} \le \lim_{n \to \infty} \frac{m_n}{n} \le a(W, \sigma) \le \lim_{n \to \infty} \frac{M_n}{n} \le \frac{M_k}{k}$$ for every $k \ge 1$.
\end{prop}

\begin{prop}\label{p:M_n-m_n bounded}
	If the graphon $W$ has finite diameter and $\sigma$ is a chip configuration on $W$ with $U^n\sigma(x) \le K$ for each $n \in \N$ and $x \in [0, 1]$, then there exists $k \in \mathbb{N}$ such that $M_n - m_n \leq k$ for each $n$.
\end{prop} 
\begin{proof}
  Let us fix $\varepsilon > 0$ and $N \in \N$ according to the equivalent definition (ii) of Theorem \ref{t:finite diam equiv conditions}. 
  Let $A=\{x\in [0,1]: u_n(x)\geq \frac{M_n+ m_n}{2}\}$. 
  Then either the measure of $A$ or the measure of the complement of $A$ is at least $\frac{1}{2}$, so we are able to use the condition (ii) for one these.
  
  \noindent \textbf{Case 1:} $\lambda(A) \ge \frac{1}{2}$, hence $\lambda(A \cup \Gamma_\varepsilon(A) \cup \Gamma^2_\varepsilon(A)\cup \dots \cup \Gamma^N_\varepsilon(A))=1$ by our choice of $\varepsilon$ and $N$. 
  
  Let us fire each vertex $m_n$ times. Then only a measure-zero set of vertices are fired more times than they should be after the first $n$ steps. After firing each vertex $m_n$ times, the chip configuration is the same as originally. Then we additionally fire each vertex the necessary number of times to fire almost all vertex $x$ exactly $u_n(x)$ times (except for the elements of the measure-zero set $\{ x : u_n(x) < m_n\}$). In this way, we get to a chip configuration that is equal to $U^n\sigma$ almost everywhere, hence it is essentially bounded by $K$.
	
	The vertices in $A$ have to be fired at least $\frac{M_n-m_n}{2}$ times additionally. 
	By the definition of $\Gamma_\varepsilon(A)$, the vertices in $\Gamma_\varepsilon(A)$ receive at least $\varepsilon$ chips if $A$ is fired. Hence after the at least $\frac{M_n - m_n}{2}$ additional firings, the vertices in $\Gamma_\varepsilon(A)$ receive at least $\varepsilon\cdot \frac{M_n - m_n}{2}$ chips. Using the bound on the configuration, and that the degree of a vertex is at most $1$, each vertex of $\Gamma_\varepsilon(A)$ has to do at least $\left\lceil\varepsilon\cdot \frac{M_n - m_n}{2}-K\right\rceil \ge \varepsilon\cdot \frac{M_n - m_n}{2}-K$ additional firings. 
	
  Continuing like this, we get that any vertex in $A \cup \Gamma_\varepsilon(A) \cup \Gamma^2_\varepsilon(A)\cup \dots \cup \Gamma^N_\varepsilon(A)$ has to do at least  $\varepsilon^N \cdot \frac{M_n-m_n}{2} - K(\varepsilon^{N-1}+\dots + \varepsilon +1)$ additional firings.
	As the vertices that fire $m_n$ times do not need any additional firing and the measure of these vertices is positive by the definition of $m_n$, we have that 
  $$\varepsilon^N \cdot \frac{M_n-m_n}{2} - K(\varepsilon^{N-1}+\dots + \varepsilon +1)\leq 0,$$
  hence $M_n-m_n \leq (\frac{1}{\varepsilon})^N\cdot 2K(\varepsilon^{N-1}+\dots + \varepsilon +1)$, which does not depend on $n$.

  \medskip
  \noindent \textbf{Case 2:} $\lambda(A) < 1/2$. Let $B=[0,1]\setminus A$. Then $\lambda(B) > 1/2$, hence $\lambda(B \cup \Gamma_\varepsilon(B) \cup \Gamma^2_\varepsilon(B)\cup \dots \cup \Gamma^N_\varepsilon(B))=1$
	
	Starting again from $\sigma$, let us fire each vertex $M_n$ times. Then the chip configuration remains the same. Now we ``inverse fire'' each vertex the necessary number of times so that in the end, the number of firings made by a vertex $x$ is $u_n(x)$ for all $x$, except those that are in the measure zero set $\{x : u_n(x) > M_n\}$. Then we reached a chip configuration $\rho$, with $\rho(x) = U^n\sigma(x)$ almost everywhere. 
  
  By definition, the vertices of $B$ have to be inverse fired at least $\frac{M_n-m_n}{2}$ times.
	By inverse firing $B$, the vertices in $\Gamma_\varepsilon(B)$ all lose at least $\varepsilon$ chips, hence after at least $\frac{M_n-m_n}{2}$ inverse firings, they lose at least $\varepsilon\cdot \frac{M_n-m_n}{2}$ chips. As originally they had at most $K$ chips and at the end they have at least 0 chips, they have to gain at least $\varepsilon\cdot \frac{M_n-m_n}{2}-K$ chips. By an inverse firing they can gain at most one chip, hence they need to perform at least $\varepsilon\cdot \frac{M_n-m_n}{2}-K$ inverse firings. 
  
  Continuing this, we get that each vertex in $B \cup \Gamma_\varepsilon(B) \cup \Gamma^2_\varepsilon(B)\cup \dots \cup \Gamma^N_\varepsilon(B)$ needs at least $\varepsilon^N \cdot \frac{M_n-m_n}{2} - K(\varepsilon^{N-1}+\dots + \varepsilon +1)$ inverse firings. As the vertices that fire $M_n$ times do not need any inverse firing and the measure of these vertices is positive by the definition of $M_n$, $\varepsilon^N \cdot \frac{M_n-m_n}{2} - K(\varepsilon^{N-1}+\dots + \varepsilon +1)\leq 0$, hence  $M_n-m_n \leq (\frac{1}{\varepsilon})^N\cdot 2K(\varepsilon^{N-1}+\dots + \varepsilon +1)$, which also does not depend on $n$. We conclude that $k = (\frac{1}{\varepsilon})^N\cdot 2K(\varepsilon^{N-1}+\dots + \varepsilon +1)$ suffices.  
\end{proof}

\begin{proof}[Proof of Theorem  \ref{t:activity_exists_with_FDC}]
  Let $\sigma$ be an arbitrary chip configuration on $W$. Since the activity of $\sigma$ exists if and only if the activity of $U\sigma$ exists, we can take a step in the parallel chip-firing, and deal with $\sigma' = U\sigma$ instead of $\sigma$. By Lemma \ref{l:boundedness after one firing}, there exists a bound $K \in \R$ such that $U^n\sigma'(x) \le K$ for each $n \in \N$ and almost all $x \in [0, 1]$. 
  By Proposition \ref{p:lim m_n/n exists} $\lim_{n \to \infty} \frac{m_n(W,\sigma')}{n}$ and $\lim_{n \to \infty} \frac{M_n(W,\sigma')}{n}$ both exist, and by Proposition \ref{p:M_n-m_n bounded} they are the same. 
  
  Since $\frac{m_n(W,\sigma')}{n} \leq \frac{u_n(W,\sigma')(x)}{n} \leq \frac{M_n(W,\sigma')}{n}$ for almost all $x$ and every $n$, it also follows that the limit $\lim_{n\to \infty} \frac{u_n(W,\sigma')(x)}{n}$ exists for almost all $x$, and equals to the value $\lim_{n \to \infty} \frac{m_n(W,\sigma')}{n}$.
\end{proof}

\section{The continuity of the activity}
\label{s:robustness}

In this section we show a ``continuity'' theorem for the activity on graphons of finite diameter.

\begin{defn}[Smooth pair]
	A pair $(W,\sigma)$, where $W$ is a graphon and $\sigma$ is a chip configuration is called \emph{smooth}, if 
\begin{equation}
    \text{for any $n\in\mathbb{N}$, $\lambda(\{x\in[0,1]: \exists k\in \mathbb{N}\quad U^n \sigma(x)= k\cdot \deg_W(x)\})=0$.}
\end{equation}
        We also say that $\sigma$ is a smooth chip configuration on $W$.
\end{defn}

The main goal of the current section is to prove the following theorem concerning the continuity of the activity.

\begin{thm} \label{thm:robustness}
Let $(W,\sigma)$ be a smooth pair and $d > 0$ where $W$ is a connected graphon with $\mindeg(W) \ge d$ and $\sigma$ is a chip configuration. Then for any $\varepsilon > 0$ there exists $\delta>0$ such that if $W'$ is a connected graphon with $\mindeg(W') \ge d$, $\mathrm{d}_{\Box}(W,W')<\delta$ and $\sigma'$ is a chip configuration with $\|\sigma-\sigma'\|_1<\delta$ then $|a(W, \sigma) - a(W', \sigma')| < \varepsilon$.
\end{thm}

\begin{rmk}
	We note that one cannot leave out the condition of $(W,\sigma)$ being a smooth pair. 
	Take for example the graphon $W\equiv 1$ and the chip configuration $\sigma \equiv 1$. Then $a(W,\sigma)=1$ as each vertex will fire once in each step. However, if we modify $\sigma$ by decreasing its values by some $\varepsilon > 0$ on each vertex then $\|\sigma - \sigma'\|_1=\varepsilon$, on the other hand, the chip configuration becomes stable, hence $a(W,\sigma')=0$.
\end{rmk}

\begin{questi}
  Would Theorem \ref{thm:robustness} remain true if we only required finite diameter for $W$ and $W'$, without asking for a common lower bound on the degrees?
  That is, is the following, stronger form of Theorem \ref{thm:robustness} true?
  
  \emph{Let $(W,\sigma)$ be a smooth pair where $W$ is a graphon of finite diameter and $\sigma$ is a chip configuration. Then for any $\varepsilon > 0$ there exists $\delta>0$ such that if $W'$ is a graphon of finite diameter with $\mathrm{d}_{\Box}(W,W')<\delta$ and $\sigma'$ is a chip configuration with $\|\sigma-\sigma'\|_1<\delta$ then $|a(W, \sigma) - a(W', \sigma')| < \varepsilon$.}
\end{questi}

We prove Theorem \ref{thm:robustness} through a series of lemmas and propositions. First, in Lemma \ref{l:m_n_tech} and \ref{l:M_n_tech} we show that if $\|u_n(W, \sigma) - u_n(W', \sigma')\|_1$ is sufficiently small then the quantities $m_n$ and $M_n$, as defined in Section \ref{ss:existence of activity}, are also close for $(W, \sigma)$ and $(W', \sigma')$. Next, in Proposition \ref{p:continuity: chip configs close after firing} and \ref{p:continuity: firing vectors close} we show that if $\mathrm{d}_\Box(W, W')$ and $\|\sigma - \sigma'\|_1$ are small then $\|U_W\sigma - U_{W'}\sigma'\|_1$ and $\|f(W, \sigma) - f(W', \sigma')\|_1$ are also small. Finally, in the proof of Theorem \ref{thm:robustness}, we put these ingredients together in an inductive argument. 

    \begin{lemma} \label{l:m_n_tech}
        Suppose $\sigma$ is a chip configuration on the graphon $W$ and $\sigma'$ is a configuration on the graphon $W'$. Let $d > 0$ and $n \in \N$ be given. If $\mindeg(W') \ge d$, $\|\sigma'\|_1 \le 2\|\sigma\|_1$ and $\|u_n(W, \sigma) - u_n(W', \sigma')\|_1 < \frac{d}{2}$, then $m_n(W', \sigma') \ge m_n(W, \sigma) - \left(\frac{2}{d} + \frac{4\|\sigma\|_1}{d^2}\right)$. 
   	\end{lemma}
    \begin{proof}
		Suppose that $m_n(W', \sigma') < m_n(W, \sigma)$, otherwise we have nothing to prove. 
        Since $\|u_n(W,\sigma)-u_n(W',\sigma')\|_1 \leq \frac{d}{2}$ and the odometer is integer-valued, for the set $A=\{x\in[0,1]: u_n(W,\sigma)(x) \neq u_n(W',\sigma')(x)\}$ we have $\lambda(A)\leq \frac{d}{2}$.
		
		Let $B=\{x\in[0,1]: u_n(W',\sigma')(x)=m_n(W',\sigma')\}$. By definition of $m_n(W',\sigma')$, $\lambda(B)>0$. Note that $B\subseteq A$ except for a measure zero set.
		For almost all $x\in B$, $\deg_{W'}(x,[0,1]\setminus A)\geq \frac{d}{2}$, since $\lambda(A)\leq \frac{d}{2}$. In the parallel chip-firing started from $\sigma'$ on $W'$, almost all vertex of $[0,1]\setminus A$ fired at least $m_n(W, \sigma) - m_n(W', \sigma')$ times more than almost all vertex in $B$, hence for almost all vertex $x\in B$, $$U^n_{W'}\sigma'(x)\geq \sigma'(x)+\left(m_n(W, \sigma) - m_n(W', \sigma')\right)\frac{d}{2}.$$ 
        
        Using Lemma \ref{l:boundedness after one firing}, $U^n_{W'} \sigma'(x) \le 1 + \frac{\|\sigma'\|_1}{d}$, and thus $U^n_{W'}\sigma'(x) \le 1 + \frac{2\|\sigma\|_1}{d}$ for almost all $x \in [0, 1]$. It follows that $m_n(W, \sigma) - m_n(W', \sigma') \le \frac{2}{d} + \frac{4\|\sigma\|_1}{d^2}$. 
	\end{proof}

    For the analogous claim about $M_n$ we need the chip configuration $\sigma'$ to be bounded.

    \begin{lemma} \label{l:M_n_tech}
        Suppose $\sigma$ is a chip configuration on the graphon $W$ and $\sigma'$ is a configuration on the graphon $W'$. Let $d > 0$, $K \in \N$ and $n \in \N$ be given. If $\mindeg(W') \ge d$, $\|u_n(W, \sigma) - u_n(W', \sigma')\|_1 < \frac{d}{2}$ and $\sigma'(x) < K$ for almost all $x$, then $M_n(W', \sigma') \le M_n(W, \sigma) + \frac{2K}{d}$. 
	\end{lemma}
    \begin{proof}
        The proof of the lemma is similar to the previous one. Using the boundedness of $\sigma'$, $M_n(W', \sigma') < \infty$, hence we can suppose that $M_n(W, \sigma) < \infty$ as well. We can also suppose that $M_n(W', \sigma') > M_n(W, \sigma)$, otherwise we have nothing to prove. Let $A$ be defined as in the proof of Lemma \ref{l:m_n_tech}, and let $B=\{x\in[0,1]: u_n(W',\sigma')(x)=M_n(W',\sigma')\}$. Then again, $\lambda(A) \le \frac{d}{2}$, $\lambda(B) > 0$, $B \subseteq A$ and for almost all $x\in B$, $\deg_{W'}(x,[0,1]\setminus A)\geq \frac{d}{2}$. 
        
        In the parallel chip-firing started from $\sigma'$ on $W'$, almost all vertex of $[0,1]\setminus A$ fired at least $M_n(W', \sigma') - M_n(W, \sigma)$ times less than almost all vertex in $B$, hence for almost all vertex $x\in B$, $$0 \le U^n_{W'}\sigma'(x) \le \sigma'(x) - \left(M_n(W', \sigma') - M_n(W, \sigma)\right)\frac{d}{2}.$$ Using $\sigma'(x) < K$, it follows that $M_n(W', \sigma') - M_n(W, \sigma) \le \frac{2K}{d}$. 
	\end{proof}
    
The following technical lemma is used many times in the proof of the next proposition.

\begin{lemma}
    \label{l:cut-distance and degree}
    For two arbitrary graphons $W$ and $W'$, $\|\deg_W - \deg_{W'}\|_1 \le 2\mathrm{d}_\Box(W, W')$. Therefore, for any $\eta > 0$, $\lambda(\{x \in [0, 1] : |\deg_W(x) - \deg_{W'}(x)| \ge \eta\}) \le \frac{2\mathrm{d}_\Box(W, W')}{\eta}$. 
\end{lemma}
\begin{proof}
    Let $A = \{x \in [0, 1] : \deg_W(x) > \deg_{W'}(x)\}$. Then 
    \begin{align*}
        \mathrm{d}_\Box(W, W') &\ge \int_A \int_0^1 W(x, y) - W'(x, y) \dy \dx = \int_A \deg_W(x) - \deg_{W'}(x) \dx \\
          &= \int_A \left|\deg_W(x) - \deg_{W'}(x)\right| \dx.
    \end{align*}
    Similarly, one can show that $\int_{A^c} |\deg_W(x) - \deg_{W'}(x)| \dx \le \mathrm{d}_\Box(W, W')$, hence the first assertion of the proposition follows. 
\end{proof}
    
We state the next proposition in a more general form than what is needed to prove Theorem \ref{thm:robustness}. As a consequence, the proof requires an extra technical step. However, the statement is simpler this way, and we believe it might be interesting on its own.

\begin{prop}\label{p:continuity: chip configs close after firing}
	Let $\sigma$ be a chip configuration on a connected graphon $W$ such that $(W, \sigma)$ is a smooth pair. Then for any $\varepsilon > 0$ there exists a $\delta > 0$ such that if $\sigma'$ is a chip configuration on a connected graphon $W'$ with $\mathrm{d}_\Box(W, W') < \delta$ and $\|\sigma - \sigma'\|_1 < \delta$, then $\|U_W \sigma - U_{W'} \sigma'\|_1 < \varepsilon$. 
\end{prop}
\begin{proof}
	Throughout the proof we use the notation $f(x) = f(W, \sigma)(x)$ for the given pair $(W, \sigma)$, and similarly $f'(x) = f(W', \sigma')(x)$ for a graphon $W'$ and a chip configuration $\sigma'$ satisfying the conditions of the proposition for a $\delta > 0$ specified later. 
	
	For $\eta > 0$ we collect the points at which the graphons or the chip configurations behave badly, so let 
	\begin{align*}
		N &= \{x \in [0, 1] : f(x) \neq f'(x)\}, \\
		U_\eta &= \left\{x \in [0, 1] : \sigma(x) \ge \frac{1}{\eta}\right\},\;\;\;
		U'_\eta = \left\{x \in [0, 1] : \sigma'(x) \ge \frac{1}{\eta}\right\},\\
		L_\eta &= \{x \in [0, 1] : \deg_W(x) \le \eta\}, \;\;\;
		L'_\eta = \{x \in [0, 1] : \deg_{W'}(x) \le \eta\},\\
		A_\eta &= N \cup U_\eta \cup U'_\eta \cup L_\eta \cup L'_\eta.
	\end{align*}
	Then we have the following.
	\begin{lemma}
		\label{l:sigma-sigma' is small}
		For any $\varepsilon' > 0$, there exist $\eta > 0$ and $\delta > 0$ such that if $W'$ is a graphon and $\sigma'$ is a chip configuration with $\mathrm{d}_\Box(W, W') < \delta$ and $\|\sigma - \sigma'\|_1 < \delta$ then $\int_{A_\eta} \sigma + \sigma' < \varepsilon'$.
	\end{lemma}
	\begin{proof}
		We first claim that for a given $\varepsilon'' > 0$ it is enough to find $\eta > 0$ and $\delta > 0$ such that if the conditions of the lemma hold then $\lambda(A_\eta) < \varepsilon''$. Indeed, since $\sigma$ is in $L^1$, for any $\varepsilon'$ there exists $\varepsilon''$ such that $\lambda(A_\eta) < \varepsilon''$ implies $\int_{A_\eta} \sigma < \frac{\varepsilon'}{3}$. Let us fix such an $\varepsilon''$. Since $\|\sigma - \sigma'\|_1 < \delta$, $\int_{A_\eta} \sigma' < \delta + \int_{A_\eta} \sigma$. Therefore, if $\eta$ and $\delta$ are small enough so that $\delta < \frac{\varepsilon'}{3}$ and $\lambda(A_\eta) < \varepsilon''$ then $\int_{A_\eta} \sigma + \sigma' < \varepsilon'$. 
		
		Let us fix an $\varepsilon'' > 0$ towards proving the above statement. First we deal with the sets $U_\eta$ and $U'_\eta$. Since $\sigma \in L^1$, there exists $\eta_0 > 0$ such that if $\eta \le \eta_0$ then $\lambda(U_{2\eta}) \le \frac{\varepsilon''}{4}$. Then $\lambda(U'_\eta \setminus U_{2\eta}) < 2\eta\delta$, since $\|\sigma - \sigma'\|_1 < \delta$. Hence, as $U_\eta \subseteq U_{2\eta}$, $\lambda(U_\eta \cup U'_\eta) < \frac{\varepsilon''}{4} + 2\eta\delta$. Clearly we can choose $\delta$ small enough so that for any $\eta \le \eta_0$, $\lambda(U_\eta \cup U'_\eta) < \frac{\varepsilon''}{3}$. 
		
		Next, we bound the measure of $L_\eta \cup L'_\eta$. Since $W$ is connected, $\lambda(\{x : \deg_W(x) = 0\}) = 0$. It follows that there exists $\eta_1 > 0$ such that if $\eta \le \eta_1$ then $\lambda(L_{2\eta}) \le \frac{\varepsilon''}{4}$. It is easy to check that $L'_\eta \subseteq L_{2\eta} \cup \{x : |\deg_W(x) - \deg_{W'}(x)| \ge \eta\}$. Therefore, using Lemma \ref{l:cut-distance and degree} and the fact that $\mathrm{d}_\Box(W, W') < \delta$, $\lambda(L'_\eta \setminus L_{2\eta}) \le \frac{\mathrm{d}_\Box(W, W')}{\eta} < \frac{\delta}{\eta}$. Then, using $L_\eta \subseteq L_{2\eta}$, for any fixed $\eta \le \eta_1$ we can choose $\delta > 0$ so that $\lambda(L_\eta \cup L'_\eta) < \frac{\varepsilon''}{3}$. 
		
		Finally, we deal with $N$. For $\zeta > 0$ let 
		$$B_\zeta = \{ x : \exists t \in \N\left(t \deg_W(x) + \zeta(t+1) < \sigma(x) < (t+1)\deg_W(x) - \zeta(t + 1) \right) \}.$$
		Since $(W, \sigma)$ is smooth and $W$ is connected, the set $B = \{ x : \deg_W(x) = 0 \text{ or } \exists t \in \N \left(\sigma(x) = t \deg_W(x)\right)\}$ is of measure $0$. Since $\bigcup_{\zeta > 0} B_\zeta = B^c$ and $\zeta < \zeta'$ implies $B_\zeta \supseteq B_{\zeta'}$, there exists $\zeta > 0$ such that $\lambda(B_\zeta) \ge 1 - \frac{\varepsilon''}{6}$. Let us fix such a $\zeta$. 
		
		Set 
		\begin{align*}
			N_\zeta^1 = \{x \in B_\zeta : f(x) > f'(x)\}, \\
			N_\zeta^2 = \{x \in B_\zeta : f(x) < f'(x)\}, 
		\end{align*}
		and note that $N \subseteq B_\zeta^c \cup N_\zeta^1 \cup N_\zeta^2$, hence it is enough to bound the measure of these two sets. 
		
		If $x \in N_\zeta^1$ then for some $t \in \N$, 
		\begin{align*}
			\sigma(x) &> t \deg_W(x) + \zeta(t + 1) \text {, and} \\
			\sigma'(x) &< t\deg_{W'}(x). 
		\end{align*}
		It follows that either $\sigma(x) \ge \sigma'(x) + \zeta$ or $\deg_{W'}(x) \ge \deg_W(x) + \zeta$, hence 
		$$N_\zeta^1 \subseteq \{x : \sigma(x) \ge \sigma'(x) + \zeta\} \cup \{x : \deg_{W'}(x) \ge \deg_W(x) + \zeta\}.$$
		
		Since $\|\sigma - \sigma'\|_1 < \delta$, $\lambda(\{x : \sigma(x) \ge \sigma'(x) + \zeta\}) < \frac{\delta}{\zeta}$. Using Lemma \ref{l:cut-distance and degree}, $\lambda(\{x : \deg_{W'}(x) \ge \deg_W(x) + \zeta\}) \le \frac{2\mathrm{d}_\Box(W, W')}{\zeta} < \frac{2\delta}{\zeta}$. 
		It follows that $\lambda(N_\zeta^1) < \frac{3\delta}{\zeta}$. 
		
		A similar argument yields that $\lambda(N_\zeta^2) < \frac{3\delta}{\zeta}$, hence, using $N \subseteq B_\zeta^c \cup N_\zeta^1 \cup N_\zeta^2$, $\lambda(N) < \frac{\varepsilon''}{6} + \frac{6\delta}{\zeta}$. It is clear, that by choosing a small enough $\delta$, $\lambda(N) < \frac{\varepsilon''}{3}$. Therefore we can complete the proof by first fixing any positive $\eta \le \min\{\eta_0, \eta_1\}$ and then choosing a small enough $\delta$, so that $\lambda(U_\eta \cup U'_\eta), \lambda(L_\eta \cup L'_\eta), \lambda(N) < \frac{\varepsilon''}{3}$.
	\end{proof}
	
	We apply Lemma \ref{l:sigma-sigma' is small} with $\varepsilon' = \frac{\varepsilon}{5}$ to obtain $\eta > 0$ and $\delta_0$, hence $\mathrm{d}_\Box(W, W') < \delta_0$, $\|\sigma-\sigma'\|_1 < \delta_0$ imply $\int_{A_\eta} \sigma + \sigma' < \frac{\varepsilon}{5}$. The final value for $\delta$ will be chosen to be less than $\delta_0$. 
	
	Now we investigate the effect of firing the vertices in $[0, 1]  \setminus {A_\eta}$ and in ${A_\eta}$ separately. So let $f = f_1 + f_2$ be the unique decomposition with $f_1(x) = 0$ if $x \in {A_\eta}$ and $f_2(x) = 0$ if $x \in [0, 1] \setminus {A_\eta}$. Let 
	\begin{align*}
		\sigma_1(x) &= \sigma(x) - f_1(x) \deg_W(x) + \int_0^1 f_1(y) W(x, y) \dy, \\
		\sigma_2(x) &= \sigma_1(x) - f_2(x) \deg_W(x) + \int_0^1 f_2(y) W(x, y) \dy.
	\end{align*}
	We define $f'_1$, $f'_2$, $\sigma'_1$ and $\sigma'_2$ analogously for $W'$ and $\sigma'$. It is straightforward to check that $\sigma_2 = U_W \sigma$ and $\sigma'_2 = U_{W'}\sigma'$, so it is enough to prove that $\|\sigma_2 - \sigma'_2\|_1 < \varepsilon$. Since 
	$$\|\sigma_2 - \sigma'_2\|_1 \le \|\sigma_2 - \sigma_1\|_1 + \|\sigma_1 - \sigma'_1\|_1 + \|\sigma'_1 - \sigma'_2\|_1,$$
	it is enough to bound these quantities. 
	
	Then 
	\begin{align*}
		 \|\sigma_2 - \sigma_1\|_1 
		 &= \int_0^1\left| -f_2(x)\deg_W(x) 
		   + \int_0^1f_2(y)W(x, y) \dy \right| \dx \\ 
		 &\le \int_0^1 f_2(x)\deg_W(x) \dx 
		   + \int_0^1\int_0^1f_2(y)W(x, y) \dy \dx \\ 
		 &\le \int_{A_\eta}\sigma(x) \dx 
		   + \int_0^1\int_0^1f_2(y)W(x, y) \dx \dy \\
		 &\le \int_{A_\eta}\sigma(x) \dx + \int_0^1f_2(y)\deg_W(y) \dy 
		 \le \int_{A_\eta}\sigma(x) \dx + \int_{A_\eta} \sigma(y) \dy \\
		 &\le \frac{2}{5}\varepsilon,
	\end{align*}
	where we used Fubini's theorem for non-negative functions to interchange the order of integration. A similar calculation shows that $\|\sigma'_2 - \sigma'_1\|_1 \le \frac{2}{5}\varepsilon$. 
	
	It remains to show that $\|\sigma_1 - \sigma_1'\|_1 < \frac{\varepsilon}{5}$, which is the tricky part of the proof. 
	\begin{align*}
		\|\sigma_1 - \sigma'_1\|_1 
		\le & \; \|\sigma - \sigma'\|_1
		  + \int_0^1\left|f_1(x)\deg_W(x) - f'_1(x) \deg_{W'}(x)\right| \dx  \\ 
		  &+ \int_0^1\left|\int_0^1f_1(y) W(x, y) \dy - \int_0^1f'_1(y) W'(x, y)\dy \right| \dx \\
		\le & \; \delta + \int_{A_\eta^c} f(x)|\deg_W(x) - \deg_{W'}(x)| \dx \\
		  &+\int_0^1\left|\int_{A_\eta^c} f(y)(W(x, y) - W'(x, y))\dy \right| \dx.
	\end{align*}
	
	Using the fact that $f(x) \le \frac{\sigma(x)}{\deg_W(x)} \le \frac{\frac{1}{\eta}}{\eta} = \frac{1}{\eta^2}$ for every $x \in A_\eta^c$ and Lemma \ref{l:cut-distance and degree}, 
	\begin{align*}
		\int_{A_\eta^c} f(x)|\deg_W(x) - \deg_{W'}(x)| \dx &\le \int_{A_\eta^c} \frac{1}{\eta^2} |\deg_W(x) - \deg_{W'}(x)| \dx \\ 
		&\le \frac{2\mathrm{d}_\Box(W, W')}{\eta^2} < \frac{2\delta}{\eta^2}.
	\end{align*}
	
	It remains to bound the integral $\int_0^1\left|\int_{A_\eta^c} f(y)(W(x, y) - W'(x, y))\dy \right| \dx$. This part is the most technical one. 
	
	Let $K$ be the largest integer with $K \le \frac{1}{\eta^2}$, and for $0 \le j \le K$, let $E_j = \{y \in [0, 1] : f(y) = j\}$. 
	For each fixed $x \in [0, 1]$, the integral $\int_{E_j} W(x, y) - W'(x, y) \dy$ can be either negative and non-negative for each $j \le K$. These give us $2^{K + 1}$ many possibilities for the sign of the integrals, thus partitioning $[0, 1]$. We encode the signs using a finite sequence $s\in\{0,1\}^{K+1}$, where $0$ corresponds to non-negative integrals and $1$ corresponds to negative ones, so let
	$$I_s = \Big\{x \in [0, 1] : \forall j \le K \; \Big(s(j) = 0 \Leftrightarrow \int_{E_j} W(x, y) - W'(x, y) \dy \ge 0\Big)\Big\}.$$ 
	Hence
	\begin{align*}
	  \int_0^1 &\Big|\int_{A_\eta^c} f(y)(W(x,y)-W'(x,y))\dy\Big|\dx \\ 
	  &\leq \int_0^1 \sum_{j \le K} \Big| \int_{E_j} j (W(x,y)-W'(x,y))\dy\Big|\dx \\ 
	  &= \sum_{s \in \{0,1\}^{K + 1}} \int_{I_s} \sum_{j \le K} (-1)^{s(j)} \int_{E_j} j (W(x,y)-W'(x,y))\dy\dx \\ 
	  &\le K \sum_{s \in \{0,1\}^{K + 1}} \sum_{j \le K} \int_{I_s} \int_{E_j} (-1)^{s(j)} (W(x,y)-W'(x,y))\dy\dx \\ 
	  &\le K(K + 1) 2^{K + 1}\mathrm{d}_\Box(W, W') < \frac{1}{\eta^2} \left(\frac{1}{\eta^2} + 1\right) 2^{\frac{1}{\eta^2} + 1}\delta,
	\end{align*}
	where we used the definition of $\mathrm{d}_\Box$ and Fubini's theorem for the integrable function $(x, y) \mapsto (-1)^{s(j)} (W(x, y) - W'(x, y))$. Hence, 
	$\|\sigma_1 - \sigma'_1\|_1 \le \delta + \frac{2\delta}{\eta^2} + \frac{1}{\eta^2} \left(\frac{1}{\eta^2} + 1\right) 2^{\frac{1}{\eta^2} + 1}\delta$. Therefore, by choosing $\delta$ small enough, we can make sure that $\|\sigma_1 - \sigma'_1\| < \frac{\varepsilon}{5}$, completing the proof of the Proposition \ref{p:continuity: chip configs close after firing}.
\end{proof}

\begin{prop}\label{p:continuity: firing vectors close}
    Suppose $(W, \sigma)$ is a smooth pair, where $W$ is a graphon with finite diameter and $\sigma$ is a chip configuration. Then for any $\varepsilon > 0$ and $d > 0$ there exists a $\delta > 0$ such that if $\sigma'$ is a chip configuration on a graphon $W'$ with $\mindeg(W') \ge d$, $\mathrm{d}_\Box(W, W') < \delta$ and $\|\sigma - \sigma'\|_1 < \delta$, then $\|f(W, \sigma) - f(W', \sigma')\|_1 < \varepsilon$. 
\end{prop}
\begin{proof}
	Let $\theta = \min\{\mindeg(W), d\}$, and apply Lemma \ref{l:sigma-sigma' is small} with $\varepsilon' = \frac{\varepsilon \theta}{2}$ to obtain $\eta > 0$ and $\delta > 0$. Then, as $f(W, \sigma)(x) \le \frac{\sigma(x)}{\deg_W(x)} \le \frac{\sigma(x)}{\theta}$ for almost all $x \in [0, 1]$, $\int_{A_\eta} f(W, \sigma) \le \frac{1}{\theta}\int_{A_\eta} \sigma < \frac{\varepsilon}{2}$. Similarly, $\int_{A_\eta} f(W', \sigma') < \frac{\varepsilon}{2}$. Therefore, $\|f(W, \sigma) - f(W', \sigma')\|_1 < \varepsilon$ for every $W'$ and $\sigma'$ with $\mathrm{d}_\Box(W, W') < \delta$ and $\|\sigma - \sigma'\| < \delta$. 
\end{proof}
    
We are now ready to prove the main theorem of this section.
    
\begin{proof}[Proof of Theorem \ref{thm:robustness}]
    Let $K = 1 + 2\frac{\|\sigma\|_1}{d}$. We first claim that it is enough to prove the theorem for chip configurations that satisfy $\sigma(x) \le K$, $\sigma'(x) \le K$ for almost all $x \in [0, 1]$. To see this, suppose that the theorem is known if $\sigma(x)\leq K, \sigma'(x)\leq K$ for almost all $x \in [0, 1]$, and let $\sigma$ and $\sigma'$ be arbitrary. Then $a(W, U_W \sigma) = a(W, \sigma)$ and $a(W', U_{W'} \sigma') = a(W', \sigma')$, and using Lemma \ref{l:boundedness after one firing}, $U_W \sigma(x) \le 1 + \frac{\|\sigma\|_1}{d}$ and $U_{W'} \sigma'(x) \le 1 + \frac{\|\sigma'\|_1}{d}$ for almost all $x \in [0, 1]$. If we set $\delta$ to be less than $\|\sigma\|_1$, then $K$ is an essential upper bound for $U_W \sigma$ and $U_{W'} \sigma'$, hence the weak version allows us to find $\delta'$ small enough so that the conclusion of the theorem holds for $U_W \sigma$ and $U_{W'} \sigma'$ instead of $\sigma$ and $\sigma'$ provided that $\|U_W \sigma - U_{W'} \sigma'\|_1 < \delta'$. Then one can use Proposition \ref{p:continuity: chip configs close after firing} with $\varepsilon = \delta'$ to find $\delta \le \delta'$ so that $\|\sigma - \sigma'\|_1 < \delta$ implies $\|U_W \sigma - U_{W'} \sigma'\|_1 < \delta'$, and we are done. 
    
    So now let $W$, $\sigma$, $d$ and $\varepsilon$ be fixed with $\sigma(x) \le K$ for almost all $x \in [0, 1]$, where $K = 1 + 2\frac{\|\sigma\|_1}{d}$. We need to show that one can find $\delta > 0$, $\delta \le \|\sigma\|_1$ such that for every $(W', \sigma')$ with $\mindeg(W') \ge d$, $\mathrm{d}_\Box(W, W') < \delta$, $\|\sigma - \sigma'\|_1 < \delta$ and $\sigma'(x) \le K$ for almost all $x \in [0, 1]$ we have $|a(W, \sigma) - a(W', \sigma')| < \varepsilon$.
    
    Using Propositions \ref{p:M_n-m_n bounded} and \ref{p:lim m_n/n exists}, $$\lim_{n \to \infty} \frac{m_n(W, \sigma)}{n} = \lim_{n \to \infty} \frac{M_n(W, \sigma)}{n}$$ and $$\frac{m_n(W, \sigma)}{n} \le a(W, \sigma) \le \frac{M_n(W, \sigma)}{n}\text{ for each $n$.}$$ Therefore we can choose $n$ large enough so that $$\frac{M_n(W, \sigma)}{n} - \frac{m_n(W, \sigma)}{n} < \frac{\varepsilon}{2}, \; \frac{2d + 4\|\sigma\|_1}{nd^2} < \frac{\varepsilon}{2}\text{ and } \frac{2K}{dn} < \frac{\varepsilon}{2},$$ where the latter two quantities come from Lemma \ref{l:m_n_tech} and Lemma \ref{l:M_n_tech}. Note that the choice of $n$ only depends on $\varepsilon, d, \sigma$ and $W$.
    
	We claim that it is enough to prove that if $\delta$ is small enough and $W', \sigma'$ satisfy the conditions of the theorem, moreover, $\sigma'(x)\leq K$ for almost all $x\in [0,1]$, then 
    \begin{equation}\label{eq:m_n,M_n_close}
    m_n(W',\sigma')\geq m_n(W, \sigma) - \left(\frac{2}{d} + \frac{4\|\sigma\|_1}{d^2}\right)  \text{  and } \\ M_n(W',\sigma')\leq M_n(W,\sigma) + \frac{2K}{d}. 
    \end{equation}
    
    Indeed in this case $\frac{m_n(W',\sigma')}{n}\geq \frac{m_n(W,\sigma)}{n}-\frac{\varepsilon}{2}\geq a(W,\sigma)-\varepsilon$, and
    $\frac{M_{n}(W,\sigma)}{n}\leq \frac{M_{n}(W,\sigma)}{n}+\frac{\varepsilon}{2}\leq a(W,\sigma)+\varepsilon$.
    Since $W'$ has finite diameter, $$a(W',\sigma')=\lim_{k\to\infty} \frac{m_k(W',\sigma')}{k}=\sup_{k}\frac{m_k(W',\sigma')}{k}\geq\frac{m_n(W',\sigma')}{n}\geq
	a(W,\sigma)-\varepsilon.$$ 
    
    Similarly, $$a(W',\sigma')=\lim_{k\to\infty} \frac{M_k(W',\sigma')}{k}=\inf_{k}\frac{M_k(W',\sigma')}{k}\leq\frac{M_n(W',\sigma')}{n}\leq
	a(W,\sigma)+\varepsilon.$$
    
    By Lemmas \ref{l:m_n_tech} and \ref{l:M_n_tech}, for \eqref{eq:m_n,M_n_close} to hold, it is enough to chose $\delta$ small enough so that if $\mathrm{d}_\Box(W, W') < \delta$ and $\|\sigma - \sigma'\|_1 < \delta$, then $\|\sigma'\|_1 \le 2\|\sigma\|_1$ and $\|u_n(W, \sigma) - u_n(W', \sigma')\|_1 < \frac{d}{2}$. The first condition is satisfied since $\delta \le \|\sigma\|_1$. To satisfy the second one, we apply Propositions \ref{p:continuity: chip configs close after firing} and \ref{p:continuity: firing vectors close} repeatedly. It is clearly enough to choose $\delta$ small enough so that for every $k \le n$, $k \ge 1$, $\|(u_k(W, \sigma) - u_{k - 1}(W, \sigma)) - (u_k(W', \sigma') - u_{k - 1}(W', \sigma'))\|_1 = \|f(W, U_W^{k - 1}\sigma) - f(W', U_{W'}^{k - 1}\sigma')\|_1 < \frac{d}{2n}$. 
    
    We first apply Proposition \ref{p:continuity: firing vectors close} to $(W, U_W^{n - 1}\sigma)$ and $\varepsilon = \frac{d}{2n}$ to get $\delta_n > 0$ so that $$\|U_W^{n - 1}\sigma - U_{W'}^{n - 1}\sigma'\|_1 < \delta_n \text{ and } \mathrm{d}_\Box(W, W') < \delta_n$$ imply $$\|f(W, U_W^{n - 1}\sigma) - f(W', U_{W'}^{n - 1}\sigma')\|_1 < \frac{d}{2n}.$$ Now let $\varepsilon_{n - 1} = \min\left\{\delta_n, \frac{d}{2n}\right\}$ and apply both Proposition \ref{p:continuity: chip configs close after firing} and Proposition \ref{p:continuity: firing vectors close} with $\varepsilon = \varepsilon_{n - 1}$ to get $\delta = \delta_{n - 1} > 0$, $\delta_{n - 1} \le \delta_n$ so that $$\|U_W^{n - 2}\sigma - U_{W'}^{n - 2}\sigma'\|_1 < \delta_{n - 1} \text{ and } \mathrm{d}_\Box(W, W') < \delta_{n - 1}$$ imply 
    \begin{align*}
        \|U_W^{n - 1}\sigma - U_{W'}^{n - 1}\sigma'\|_1 < \varepsilon_{n - 1} \le \delta_n \text{ and } \\ 
        \|f(W, U_W^{n - 2}\sigma) - f(W', U_{W'}^{n - 2}\sigma')\|_1 < \varepsilon_{n - 1} \le \frac{d}{2n}.
    \end{align*}
    
    By continuing downwards in a similar fashion, we can arrive at $\delta = \delta_1 > 0$ such that $\|\sigma - \sigma'\|_1 < \delta$ and $\mathrm{d}_\Box(W, W') < \delta$ imply $\|U_W^{k - 1}\sigma - U_{W'}^{k - 1} \sigma'\|_1 \le \delta_k$ and thus $\|f(W, U_W^{k - 1}\sigma) - f(W', U_{W'}^{k - 1}\sigma')\|_1 < \frac{d}{2n}$ for each $k \le n$, $k \ge 1$. Therefore $\|u_n(W, \sigma) - u_n(W', \sigma')\|_1 < \frac{d}{2}$, and the proof of the theorem is finally complete. 
\end{proof}

The next proposition shows that $(W,\sigma)$ being a smooth pair is not a very strong condition.

\begin{prop}\label{prop:bad_chip_conf_countable}
	For any chip configuration $\sigma:[0,1]\to \mathbb{R}$ and any connected graphon $W$, the set $\{\mu\in[0,1]: (W, \sigma + \mu \cdot \deg_W) \text{ is not a smooth pair}\}$ is countable. 
\end{prop}
\begin{proof}
	Fix the chip configuration $\sigma$ and for any $\mu \in [0, 1]$, let us use the notation $\sigma_\mu = \sigma + \mu \cdot \deg_W$. For a $\mu \in [0, 1]$ and $n,\ell, k\in\mathbb{N}$, let
	\begin{align*}
	bad(\mu,n,\ell, k) = \{x \in [0, 1]: & \; u_n(W, \sigma_\mu)(x) = \ell, \; U^n\sigma_\mu(x)=k \cdot \deg_W(x)\}.
	\end{align*}
	
	It is clear from the definition that $bad(\mu, n, \ell, k)$ is measurable for each $\mu \in [0, 1]$ and $n, \ell, k \in \N$. Let us fix $n, \ell, k \in \N$, we now show that if $\mu' \neq \mu$ then $\lambda(bad(\mu, n, \ell, k) \cap bad(\mu', n, \ell, k))  = 0$. 
    Suppose that $\mu' > \mu$, then by Lemma \ref{l:more_chips_fire_more}, $u_n(W, \sigma_{\mu'})(x) \ge u_n(W, \sigma_{\mu})(x)$ for almost all $x \in [0, 1]$. Then for an $x \in bad(\mu, n, \ell, k) \cap bad(\mu', n, \ell, k)$ with $\deg_W(x) > 0$, 
    \begin{align*}
        U^n &\sigma_{\mu'}(x) \\ 
        &= \sigma_{\mu'}(x) - u_n(W, \sigma_{\mu'})(x) \cdot \deg_W(x) + \int_0^1 u_n(W, \sigma_{\mu'})(y) W(x, y) \dy \\
        &= \sigma_{\mu'}(x) - u_n(W, \sigma_{\mu})(x) \cdot \deg_W(x) + \int_0^1 u_n(W, \sigma_{\mu'})(y) W(x, y) \dy \\
        &> \sigma_{\mu}(x) - u_n(W, \sigma_{\mu})(x) \cdot \deg_W(x) + \int_0^1 u_n(W, \sigma_{\mu})(y) W(x, y) \dy \\
        &= U^n\sigma_\mu(x),
    \end{align*}
    contradicting the fact that $U^n\sigma_{\mu'}(x) = U^n\sigma_\mu(x) = k\cdot \deg_W(x)$. Hence indeed, almost all $x \in [0, 1]$ cannot be in both $bad(\mu, n, \ell, k)$ and $bad(\mu', n, \ell, k)$. 
    
    If for fixed $n, \ell, k \in \N$ uncountably many $\mu$ exists with $\lambda(bad(\mu, n, \ell, k)) > 0$, then for infinitely many of those, $\lambda(bad(\mu, n, \ell, k)) > \varepsilon$ for some $\varepsilon > 0$. By taking at least $\frac{1}{\varepsilon} + 1$ sets of those, two will intersect in a set of positive measure, a contradiction. We conclude that for fixed $n, \ell, k \in \N$, only countably many $\mu$ exists with the property that $bad(\mu, n, \ell, k)$ is of positive measure. Therefore all, but countably many $\mu$ has the property that $\lambda(bad(\mu, n, \ell, k)) = 0$ for every $n, \ell, k \in \N$, hence for all, but countably many $\mu \in [0, 1]$, $(W, \sigma_\mu)$ is a smooth pair.
\end{proof}

\section{The Devil's staircase phenomenon}

In this section we use our previous results to prove the Devil's staircase phenomenon in some situations. First, we prove that under mild conditions, the activity diagram of a chip configuration on an Erd\H os--R\'enyi random graph is close to a Devil's staircase with high probability. Then we show a one-parameter family of random chip configurations on Erd\H os--R\'enyi random graphs that exhibit the Devil's staircase phenomenon with high probability.
Let $C_p$ denote the graphon with $C_p(x,y)=p$ for all $x,y\in [0,1]$. 

\subsection{A sufficient condition for the Devil's staircase phenomenon on $C_p$}

Here we give a sufficient condition for the activity diagram of a chip configuration on $C_p$ to be a Devil's staircase. We deduce the sufficient condition from the analogous theorem of Levine \cite{Lionel_parallel}, which concerns $C_1$. Let us first note the relationship of the activity on $C_1$ and on $C_p$.

\begin{prop}\label{prop:relationship_of_C_1_and_C_p}
For any $\sigma$ and $0 < p\leq 1$, $a(C_p,\sigma)=a(C_1,\frac{\sigma}{p})$.  
\end{prop}
\begin{proof}
It is enough to show that for each $n$, $\frac{1}{p}U_{C_p}^n(\sigma)=U_{C_1}^n (\frac{\sigma}{p})$. This implies that 
$\lfloor \frac{1}{p}U_{C_p}^n(\sigma)(x) \rfloor = \lfloor U_{C_1}^n(\frac{\sigma}{p})(x)\rfloor$ for each $n$, hence the odometers are the same.

Proving $\frac{1}{p}U_{C_p}^n(\sigma)=U_{C_1}^n (\frac{\sigma}{p})$ is straightforward by induction on $n$.
\end{proof}

Now we can use the results of \cite{Lionel_parallel} that gives a sufficient condition for the activity diagram of a chip configuration on the graphon $C_1$ to be a Devil's staircase. (We note that \cite{Lionel_parallel} uses a different terminology, in particular, it does not refer to graphons.) 

For \cite{Lionel_parallel}, a \emph{generalized chip configuration} is a measurable function $\sigma:[0,1]\to [0,\infty)$ (hence every chip configuration on a graphon as defined in the current paper is a generalized chip configuration as defined in \cite{Lionel_parallel}). The update operator $U$ defined in equation (8) of \cite{Lionel_parallel} coincides with the parallel update rule for the graphon $C_1$ if $\sigma(x)< 2$ for each $x\in[0,1]$. In \cite{Lionel_parallel}, the activity of a generalized chip configuration on $C_1$ is defined as $\lim_{n\to \infty} \frac{\beta_n(\sigma)}{n}$ (if it exists), where $\beta_n(\sigma)=\sum_{i=0}^{n-1}\lambda(\{x: U^i\sigma(x)\geq 1\})$. One can make the obvious generalization and for an arbitrary graphon $W$ and chip configuration $\sigma$. Set 
$$\beta_n(W,\sigma)=\sum_{i=0}^{n-1}\lambda(\{x: U^i\sigma(x)\geq \deg_W(x)\}).$$

\begin{prop}
  \label{p:two activity def}
  If $W$ has finite diameter and $\sigma(x) < 2\deg_W(x)$ for almost all $x\in [0,1]$, then the two definitions of the activity coincide, that is, $a(W,\sigma) = \lim_{n\to \infty} \frac{\beta_n(W,\sigma)}{n}$.
\end{prop}
\begin{proof}
  By Theorem \ref{t:activity_exists_with_FDC}, if $W$ has finite diameter, then there exist $a(W,\sigma)$ such that $\lim_{n\to\infty}\frac{u_n(x)}{n}=a(W,\sigma)$ for almost all $x\in[0,1]$. Notice that $\sigma(x)< 2\deg_W(x)$ for almost all $x\in [0,1]$ implies $U^i\sigma(x)< 2\deg_W(x)$ for all $i \in \N$ and almost all $x\in [0,1]$, and hence $u_n(x)=|\{i\in \mathbb{N}:0\leq i < n, U^i\sigma(x)\geq \deg_W(x)\}|$ for almost all $x\in[0,1]$. Therefore  
  $$\beta_n(W,\sigma)=\int_0^1 u_n(x)\dx.$$

As $\lim_{n\to\infty}\frac{u_n(x)}{n}=a(W,\sigma)$ for almost all $x\in[0,1]$, for any $\varepsilon$, we can choose $n_0$ such that for any $n\geq n_0$, for $A_n=\{x\in[0,1]: |\frac{u_n(x)}{n}-a(W,\sigma)|\leq \varepsilon\}$, we have $\lambda(A_n)\geq 1-\varepsilon$.

Then for $n\geq n_0$, $$\frac{\beta_n(W,\sigma)}{n}=\int_0^1 \frac{u_n(x)}{n}\dx\leq \int_{A_n} (a(W,\sigma) + \varepsilon) \dx + \int_{[0,1]\setminus A_n} 1 \dx\leq a(W,\sigma) + 2\varepsilon.$$ Hence $\lim_{n\to \infty}\frac{\beta_n(W,\sigma)}{n}\leq a(W,\sigma)$.

Similarly, for $n\geq n_0$, $$\frac{\beta_n(W,\sigma)}{n}=\int_0^1 \frac{u_n(x)}{n}\dx\geq \int_{A_n} (a(W,\sigma) - \varepsilon) \dx \geq (1-\varepsilon) (a(W,\sigma)-\varepsilon).$$ Hence $\lim_{n\to \infty}\frac{\beta_n(W,\sigma)}{n}\geq a(W,\sigma)$, and the proof is complete.
\end{proof}

Now we collect the results from \cite{Lionel_parallel} that we need. The statements and arguments that follow are all present in \cite{Lionel_parallel}, but not everything is in a form convenient for us, so we repeat some of the arguments of that paper. We call a chip configuration $\sigma$ on $C_1$ \emph{preconfined} if $\sigma(x) < 2$ for all $x \in [0, 1]$. To each preconfined $\sigma$, let us define the function $f_\sigma : \R \to \R$ the following way. For $x \in [0, 1]$, let 
\begin{equation}
  \label{e:f def}
    f_\sigma(x) = \lambda(\{v : \sigma(v) \ge 1\}) + \lambda(\{v : \sigma(v) \in [1 - x, 1) \cup [2 - x, 2)\}).
\end{equation}

It is easy to check that $f$ is an increasing function with $f_\sigma(1) = f_\sigma(0) + 1$. Hence there is a unique extension of $f_\sigma$ to $\R$ as an increasing function, which we also denote by $f_\sigma$, that satisfies $f_\sigma(x + 1) = f_\sigma(x) + 1$. If $f_\sigma$ is continuous then it has a well-defined \emph{Poincar\'e rotation number} 
$$
  \rho(f_\sigma) = \lim_{n \to \infty} \frac{f_\sigma^n(x)}{n},
$$ 
which is independent of $x$, see \cite{Lionel_parallel}. 

\begin{lemma}[{\cite[Lemma 6]{Lionel_parallel}}]
  If $\sigma$ is preconfined and $f_\sigma$ is continuous then $a(C_1,\sigma) = \rho(f_\sigma)$. 
\end{lemma}

It is easy to check (and the computation can also be found in \cite{Lionel_parallel}) that if $y \in \R$ is given such that $\sigma + y$ is also a preconfined chip configuration on $C_1$ (that is, $0 \le \sigma + y < 2$), then $f_{\sigma + y}(x - y) = f_\sigma(x)$. As stated also in \cite{Lionel_parallel}, conjugating by the homeomorphism $R_y : \R \to \R$ defined by $R_y(x) = x + y$ does not change the rotation number. Then, for any $y \in \R$, $\rho(f_{\sigma + y}) = \rho(R_{y}(f_{\sigma + y}(R_{-y})))$. The function inside is $x \mapsto R_{y}(f_{\sigma + y}(x - y)) = R_{y}(f_{\sigma}(x))$, hence, using also the previous lemma, we have the following. 

\begin{lemma}
  \label{l:a(sigma + y) = rho(R_y circ f_sigma)}
  If $\sigma$ is a chip configuration on $C_1$, $y \in \R$ such that $\sigma$ and $\sigma + y$ are preconfined, and both $f_\sigma$ and $f_{\sigma+y}$ are continuous, then $a(C_1,\sigma + y) = \rho(R_{y} \circ f_{\sigma})$. 
\end{lemma}

Let $\sigma$ be a stable chip configuration on $C_1$, that is, $\sigma(v) < 1$ for almost all $v \in [0, 1]$. Note that for $y \in [0, 1]$, $\sigma$ and $\sigma + y$ are both preconfined. Now define $\Phi_{\sigma,y}:\mathbb{R}\to\mathbb{R}$ by
$$\Phi_{\sigma,y}(x)=\lceil x\rceil - \lambda(\{v\in [0,1]: \sigma(v)<\lceil x \rceil - x\}) + y.$$

It is easy to check that $\Phi_{\sigma, y}$ is continuous if $\lambda(\{v : \sigma(v) = x\}) = 0$ for all $x \in [0, 1]$, and also that $\Phi_{\sigma, y}(x) = R_y(f_{\sigma}(x))$, hence 
$$s(C_1, \sigma)(y) = a(C_1, \sigma + y) = \rho(\Phi_{\sigma, y}),$$
where $s$ is the activity diagram as defined in Section \ref{ss:parallel cf on graphons}.
Since $\Phi_{\sigma,y}(x+1)=\Phi_{\sigma,y}(x)+1$ for every $x \in \R$, it makes sense to denote by $\overline{\Phi}_{\sigma,y}$ the corresponding map from $\R/\Z = S^1$ to $S^1$. 

\begin{thm}[{\cite[Proposition 10]{Lionel_parallel}}]
If $\sigma(v) < 1$ for almost all $v\in[0,1]$, $\lambda(\{v: \sigma(v)=c\})=0$ for each $c\in \mathbb{R}$ and $\overline{\Phi}_{\sigma,y}^q\neq Id$ for any $q \in \N \setminus \{0\}$, then $s(C_1,\sigma)$ is a Devil's staircase. Moreover, if $\alpha$ is irrational, then $s(C_1,\sigma)^{-1}(\alpha)$ is a point, and if $y$ is rational then $s(C_1,\sigma)^{-1}(y)$ is an interval of positive length.
\end{thm}

Applying Proposition \ref{prop:relationship_of_C_1_and_C_p}, we get the following corollary for $C_p$. 

\begin{thm}\label{thm:C_p_Devils_staircase}
For some $0<p\leq 1$,
if $\sigma(v) < p$ for almost all $v\in[0,1]$, $\lambda(\{v: \sigma(v)=c\})=0$ for each $c\in \mathbb{R}$ and $\overline{\Phi}_{\frac{1}{p}\sigma,y}^q\neq Id$ for any $q \in \N \setminus \{0\}$, then $s(C_p,\sigma)$ is a Devil's staircase. Moreover, if $\alpha$ is irrational, then $s(C_p,\sigma)^{-1}(\alpha)$ is a point, and if $y$ is rational then $s(C_p,\sigma)^{-1}(y)$ is an interval of positive length.
\end{thm}

\subsection{Random graphs}

We give a sufficient condition for the activity diagrams of random graphs converging to a Devil's staircase. We collected the necessary background on random graphs in Appendix \ref{app:random_graphs}. 

First, we need the following result. 

\begin{thm}\cite[Theorem 11.32]{Lovasz_large_graphs}, \cite[Corollary 2.6]{Lovasz-Szegedy} \label{t:random_graphs_conv}
If $G_n = G(n, p)$ is a sequence of Erd\H{o}s--R\'enyi graphs then $G_n \to C_p$ with probability $1$.
\end{thm}

We note here, that the referenced papers use the unlabeled cut distance to prove the above theorem. However, as noted in Section \ref{ss:parallel cf on graphons}, $\delta_\Box(G_n, C_p) = \mathrm{d}_\Box(G_n, C_p)$ for each graph $G_n$, hence the theorem remains true if the convergence is understood using the labeled cut distance. We also note that by definition, $\mathrm{d}_\Box(G_n, C_p) = \mathrm{d}_\Box(W_{G_n}, C_p)$, hence the convergence also holds for the graphon version of the graphs.

To deal with the convergence of chip configurations on graphs, we do the following: for a graph $G$ with vertices labeled $v_1, \dots, v_n$, we define the graphon version $\tilde{\sigma} : [0, 1] \to \R$ of a chip configuration $\sigma$ by $\tilde{\sigma}(x) = \frac{1}{n}\sigma(v_i)$ if $\frac{i-1}{n}\leq x < \frac{i}{n}$.
For a sequence of chip configurations $(\sigma_n)_n$ such that $\sigma_n$ lives on the graph $G_n$, we say that they are convergent if the sequence of graphon versions $(\tilde{\sigma}_n)_n$ is convergent in the $\|.\|_1$ norm.

\begin{thm} \label{t:ER_Devil}
Suppose that $(G_n)_n$ is a sequence of (labeled) Erd\H os--R\'enyi random graphs, where $G_n= G(n,p)$, $0<p\leq 1$, $\sigma_n$ is a chip configuration on $G_n$ for each $n$, and $\|\tilde{\sigma}_n - \sigma\|_1 \to 0$ for some chip configuration $\sigma$ on $C_p$ such that $\sigma(x) < p$ for almost all $x \in [0, 1]$. Moreover, suppose that $\lambda(\{x\in[0,1]: \sigma(x)=c\})=0$ for each $c\in \mathbb{R}$ and $\overline{\Phi}_{(1/p)\sigma,y}^q\neq Id$ for any $y \in [0, 1]$ and $q \in \N$. Then
with probability $1$, the sequence of activity diagrams $(s(G_n, \sigma_n))_n$ converges uniformly to the Devil's staircase $s(C_p,\sigma)$.
\end{thm} 
\begin{proof}
Since the activity diagram $s(C_p, \sigma)$ is a Devil's staircase by Theorem \ref{thm:C_p_Devils_staircase}, it is enough to prove that with probability 1, $s(G_n, \sigma_n) \to s(C_p, \sigma)$ uniformly as $n\to\infty$. 

We would like to apply Theorem \ref{thm:robustness} for the graphon $C_p$ and the chip configuration $\sigma_y = \sigma + yp\mathbf{1}_{[0,1]}$.

As noted above, with probability 1, $\mathrm{d}_\Box(G_n, C_p) \to 0$. By our assumption, the graphon versions $\tilde{\sigma}_1, \tilde{\sigma}_2, \dots$ converge to $\sigma$ in $\|.\|_1$. 

We need to show that $C_p$ has finite diameter, but this is trivial, since for any set $A$ with $\lambda(A) > 0$, $\varepsilon = p\lambda(A)$ works to show that $\Gamma_\varepsilon(A)=[0,1]$.

We claim that $(C_p, \sigma_y)$ is a smooth pair for any value of $y \in[0,1]$. First notice that  $\sigma_y(x)<2p$ for each $y \in[0,1]$ and $x$. This implies $U^i\sigma_y(x)<2p$ for each $i$ by induction. (Indeed, in any step, any vertex fires at most once. Hence any vertex can gain at most $p$ chips in a step. But if a vertex already had at least $p$ chips, then it also fires, hence its number of chips does not increase.) We now claim that $\lambda(\{x: U^n\sigma_y(x)=c \})=0$ for any $n \in \N$, $y \in [0, 1]$ and $c \in \R$. Let us fix $y \in [0, 1]$ and define $c_i=\lambda(\{x:U^i\sigma_y(x)\geq p \})$. Our claim for $n = 0$ is a condition of the theorem, and for $n > 0$ we have 
\begin{align*}
  \{x: U^n\sigma_y(x)=c \} = \;
  &\{x : U^{n - 1}\sigma_y(x) < p \text{ and } U^{n - 1}\sigma_y(x) = c - c_{n - 1}\} \cup \\
  &\{x : U^{n - 1}\sigma_y(x) \ge p \text{ and } U^{n - 1}\sigma_y(x) = c - c_{n - 1} + p\}.
\end{align*} 
One can easily show by induction on $n$, using the above equality, that $\{x : U^n\sigma_y(x) = c\}$ is indeed a set of measure $0$ for each $n$, $y$ and $c$.

Fix $d<p$ and $\varepsilon>0$. We show that 
\begin{equation}\label{eq:activity_close}
\begin{split}
\text{with probability 1 there exists $n_0$ such that if $n\geq n_0$ and $y \in [0, 1]$, then }\\
|a(C_p,\sigma_y)-a(G_n,\sigma_n + y\deg_{G_n})|<\varepsilon.
\end{split}
\end{equation}

For any $y \in [0, 1]$, we can apply Theorem \ref{thm:robustness} for the pair $(C_p, \sigma_y)$ with $\varepsilon$ and $d$ to get $\delta$.
Let $U_n=W_{G_n}$ be the graphon corresponding to $G_n$. Notice that the graphon version of $\sigma_n + y\deg_{G_n}$ is $\tilde{\sigma}_n + y\deg_{U_n}$. It is also easy to see that $a(G_n,\sigma_n + y\deg_{G_n}) = a(U_n, \tilde{\sigma}_n+y\deg_{U_n})$, since the chip-firings on $G_n$ and on $U_n$ correspond to each other. Therefore to get \eqref{eq:activity_close} using Theorem \ref{thm:robustness} and Theorem \ref{t:finite diam equiv conditions}, we need to show that with probability 1, there exists $n_0$ such that for $n\geq n_0$, $U_n$ is connected, has minimal degree at least $d$, $\mathrm{d}_\Box(C_p,U_n) < \delta$, and for any $y\in [0, 1]$,  $\|\sigma_y-(\tilde{\sigma}_n+y\deg_{U_n})\|_1<\delta$. $U_n$ has degree at least $d$ for each point if and only if $G_n$ has degree at least $dn$. Hence by Proposition \ref{prop:min_degree_in_random_graph_seq} with probability 1 there exists an index $n_1$ such that for each $n\geq n_1$, the graphon $U_n$ has $\mindeg(U_n)\geq d$. If $G_n$ is connected, then $U_n$ is connected, hence by Proposition \ref{prop:random_graph_connected}, there exists $n_2$ such that for $n\geq n_2$, $U_n$ is connected.
With probability 1, $\mathrm{d}_\Box(C_p,U_n)$ tends to 0 by the remark after Theorem \ref{t:random_graphs_conv}, hence there exists an index $n_3$ such that $\mathrm{d}_\Box(C_p,U_n) < \delta$ for $n \ge n_3$. Now $$\|\sigma_y-(\tilde{\sigma}_n + y\deg_{U_n})\|_1\leq\|\sigma - \tilde{\sigma}_n\|_1+\|yp\mathbf{1}_{[0,1]}- y\deg_{U_n}\|_1.$$
Here $\|\sigma - \tilde{\sigma}_n\|_1$ tends to 0 by the assumptions of the theorem, hence it is below $\delta/2$ for $n \ge n_4$ for some index $n_4$. For a fixed $x \in [0, 1]$ let $v$ be the vertex of $G_n$ such that $x$ belongs to the part of $[0,1]$ corresponding to $v$. Then, using that $y \le 1$,  
\begin{align*}
  &\mathbb{P}\left[\text{$\exists y \in [0, 1]$}\left(|yp-y\deg_{U_n}(x)|>\frac{\delta}{2}\right)\right]\leq \mathbb{P}\left[|p-\deg_{U_n}(x)|>\frac{\delta}{2}\right] \\ &= \mathbb{P}\left[\left|p-\frac{\deg_{G_n}(v)}{n}\right|>\frac{\delta}{2}\right]=\mathbb{P}\left[|pn-\deg_{G_n}(v)|>\frac{\delta n}{2}\right]\leq 2e^{-\frac{\delta^2 n}{8}}    
\end{align*}
by Claim \ref{cl:azuma_for_degree}. As $G_n$ has $n$ vertices, $\mathbb{P}[\exists x\in[0,1]: |p-\deg_{U_n}(x)|>\delta/2]	\leq 2ne^{-\frac{\delta^2 n}{8}}$.
Since $\sum_{n \ge 1} 2ne^{-\frac{\delta^2 n}{8}} < \infty$, by the Borel--Cantelli lemma, with probability 1 there exists $n_5$ such that $|p-\deg_{U_n}(x)|\le \delta/2$ for all $n\geq n_5$ and all $x\in[0,1]$. 
Then for $n\geq n_5$ we have $|yp-y\deg_{U_n}(x)|\le \delta/2$ for all $x \in [0, 1]$ and $y \in [0, 1]$, hence
\begin{align*}
  \|yp\mathbf{1}_{[0,1]}- y\deg_{U_n}\|_1 = & \int_0^1|yp- y\deg_{U_n}(x)|\dx \leq \delta/2.
\end{align*}

Now with probability 1 the index $n_0 = \max\{n_1, n_2, n_3, n_4, n_5\}$ exists and the conditions of Theorem \ref{thm:robustness} are satisfied with $W' = U_n$ and $\sigma' = \tilde{\sigma}_n + y\deg_{U_n}$ for $n \ge n_0$ and $y \in [0, 1]$.
We conclude that for each $\varepsilon>0$, with probability 1, there exists an index $n_0$ with $|a(C_p,\sigma_y)-a(G_n,\sigma_n + y\deg_{G_n})|<\varepsilon$ for $n\geq n_0$ and $y \in [0, 1]$. 
Taking a sequence of $\varepsilon$ values tending to 0, we conclude that with probability 1, $s(G_n, \sigma_n)$ tends to $s(C_p, \sigma)$ uniformly, therefore the proof is complete.
\end{proof}

\subsection{Geometric random chip configurations}

We show a concrete example where the activities of a one parameter family of chip configurations on a random graph give a Devil's staircase with high probability. We will again take an Erd\H os--R\'enyi random graph, but this time we put a random number of chips on the vertices independently following geometric distribution, and look at how the activity changes if we increase the mean of the geometric distribution.

Let $G_n=G(n,p)$ for some $0<p\leq 1$. Suppose that for $v\in V(G_n)$, $\sigma^\mu_n(v)\sim Geometric(\frac{1}{1 + \mu n})$ independently for some $\mu > 0$. Here we mean the geometric distribution as $P(\sigma_n^\mu(v) = k) = (\mu n)^k / (1+\mu n)^{k+1}$ for $k\geq 0$. Note that this way, the expected value $\mathbb{E}\sigma^\mu_n(v)= \mu n$. Let us relabel the vertices such that $\sigma^\mu_n(v_1)\leq \sigma^\mu_n(v_2)\leq \dots$, and let us denote by $\tilde{\sigma}^\mu_n$ the corresponding chip configuration on the graphon $W_{G_n}$. Let us take these random chip configurations independently for each $n\in \mathbb{N}$.
For different values of $\mu$, we couple the random chip configurations in the following way. For each vertex $v$, we independently generate countably many independent uniform random variables between $0$ and $1$. For some value $\mu$, we put $k$ chips on $v$ if the first $k$ of its random variables are between $\frac{1}{1+\mu n}$ and $1$, and the $(k+1)^{th}$ is between $0$ and $\frac{1}{1+\mu n}$. This way we obtain independent $Geometric(\frac{1}{1 + \mu n})$ random variables for each vertex.

We show the following. 
\begin{thm}\label{t:geometric_staircase}
Suppose that $G_n= G(n,p)$ for $0<p\leq 1$, and $\sigma^\mu_n$ is a chip configuration where the number of chips on each vertex is an independent Geometric random variable with mean $\mu n$, coupled for different values of $\mu$ as above. Then
with probability one, the sequence of functions $\mu \mapsto a(G_n, \sigma^\mu_n)$ converges pointwise to a Devil's staircase on the interval $\mu\in[0,\frac{p}{\log 2}]$.
\end{thm}

To prove this theorem, we need to find out the limit of the chip configurations.
Let $\sigma^\mu(v)=-\mu \log(1-v)$ be a chip configuration on $C_p$, where by $\log$ we mean the natural logarithm.
We will show the following: 
\begin{lemma} \label{l:sigma_n_tends_to_sigma}
For any fixed $\mu > 0$, $\|\tilde{\sigma}^\mu_n - \sigma^\mu\|_1 \to 0$ with probability $1$ as $n\to\infty$. 
\end{lemma}
We will prove the lemma later. As $G_n \to C_p$ with probability 1, one needs to examine the behaviour of the activity of $\sigma^\mu$ on $C_p$. Unfortunately we cannot directly apply Theorem \ref{thm:C_p_Devils_staircase} here, as we do not talk about activity diagrams, but a different one-parameter family of chip configurations. However, we can still show the following.

\begin{lemma}\label{l:limit_geom_devil's_staircase}
  The map $\mu \mapsto a(C_p,\sigma^\mu)$ is a Devil's staircase on $[0, \frac{p}{\log 2}]$.
\end{lemma}
\begin{proof}
  The chip configuration $\sigma^\mu$ is unbounded, but $U\sigma^\mu$ is bounded, and since $a(\sigma^\mu) = a(U\sigma^\mu)$, it is enough to deal with the latter. To calculate $U\sigma^\mu$, let us denote by $\{x\}_p$ the $p$-fractional part of $x \in \R$, that is, the unique number in $[0, p)$ with the property that $x + kp = \{x\}_p$ for some $k \in \Z$. Then
  \begin{align}
  \label{e:U sigma mu}
  U\sigma^\mu(v) = \{-\mu \log(1 - v)\}_p + p\sum_{n \ge 1} \lambda(\{u : -\mu \log(1 - u) \ge np\}).
  \end{align}
  Now, as it is easier to handle monotone increasing chip configurations and $U\sigma^\mu$ is not increasing (as a function $v \mapsto (U\sigma^\mu)(v)$), we try to rearrange it to an increasing chip configuration $\overline{\sigma}^\mu$ with the property that 
  \begin{align}
    \label{e:changing chip-dist}
    \text{$\lambda(\{v : U\sigma^\mu(v) < x\}) = \lambda(\{v : \overline{\sigma}^\mu(v) < x\})$ for every $x \in \R$}.
  \end{align}
  Clearly, if \eqref{e:changing chip-dist} holds, the analogous statement will hold for $U^k(U\sigma^\mu(v))$ and $U^k(\overline{\sigma}^\mu)$, hence, by Proposition \ref{p:two activity def}, $a(\overline{\sigma}^\mu) = a(U\sigma^\mu) = a(\sigma^\mu)$. 
  
To define an increasing $\overline{\sigma}^\mu$ satisfying \eqref{e:changing chip-dist} , our only option is that $\overline{\sigma}^\mu(v) = x$ if and only if $\lambda(\{u : U\sigma^\mu(u) < x\}) = v$. Since 
  \begin{align}
    \label{e:level measure}
    \lambda(\{u : -\mu\log(1 - u) \ge y\}) = e^{-\frac{y}{\mu}},
  \end{align}
  \begin{align}
    \label{e:y(mu) def}
    p\sum_{n \ge 1} \lambda(\{u : -\mu \log(1 - u) \ge np\}) = \frac{pe^{-\frac{p}{\mu}}}{1 - e^{-\frac{p}{\mu}}} 
    =: y(\mu).
  \end{align}
  From \eqref{e:level measure} we also have for $x \in [0, p)$ that
  \begin{align*}
    \lambda(\{v : \{-\mu \log(1 - v)\}_p < x\}) &= \lambda\bigg(\bigg\{v : -\mu \log(1 - v) \in \bigcup_{n \ge 0} [np, np + x)\bigg\}\bigg) \\ &= \sum_{n \ge 0} e^{-\frac{np}{\mu}} - e^{-\frac{np + x}{\mu}} = \frac{1 - e^{-\frac{x}{\mu}}}{1 - e^{-\frac{p}{\mu}}}.
  \end{align*}
  Hence, using also \eqref{e:U sigma mu} and \eqref{e:y(mu) def},
  \begin{align*}
    \lambda(\{u : U\sigma^\mu(u) < x + y(\mu)\}) = \frac{1 - e^{-\frac{x}{\mu}}}{1 - e^{-\frac{p}{\mu}}} = v \Leftrightarrow x = -\mu\log(1 - v + ve^{-\frac{p}{\mu}}),
  \end{align*}
  thus
  \begin{align*}
    \overline{\sigma}^\mu(v) = -\mu\log(1 - v + ve^{-\frac{p}{\mu}}) + y(\mu) = -\mu\log(1 - v + ve^{-\frac{p}{\mu}}) + \frac{pe^{-\frac{p}{\mu}}}{1 - e^{-\frac{p}{\mu}}}.
  \end{align*}
  
  It is easy to check that it satisfies \eqref{e:changing chip-dist}, and also that $\overline{\sigma}^\mu(v) \le 2p$ if $v \in [0, 1]$ and $\mu \in [0, \frac{p}{\log(2)}]$. Since $-\mu\log(1 - v + ve^{-\frac{p}{\mu}}) \ge 0$ if $v \in [0, 1]$, we can apply Proposition \ref{prop:relationship_of_C_1_and_C_p} and Lemma \ref{l:a(sigma + y) = rho(R_y circ f_sigma)} with $\sigma = \overline{\sigma}^\mu - y(\mu)$ and $y = y(\mu)$ to get that $$a(\overline{\sigma}^\mu) = \rho\big(R_{\frac{y(\mu)}{p}}\big(f_{\frac{\overline{\sigma}^\mu - y(\mu)}{p}}\big)\big),$$
  where $f_{\frac{\overline{\sigma}^\mu - y(\mu)}{p}}$ is defined as in \eqref{e:f def}.
  
  With the notation $f^\mu = R_{\frac{y(\mu)}{p}}\big(f_{\frac{\overline{\sigma}^\mu - y(\mu)}{p}}\big)$, our task is to show that $\mu\mapsto \rho(f^\mu)$ is a Devil's staircase. For $x \in [0, 1]$, 
  \begin{align*}
    f^\mu(x) &= R_{\frac{y(\mu)}{p}}\big(f_{\frac{\overline{\sigma}^\mu - y(\mu)}{p}}\big)(x) \\ &= \frac{y(\mu)}{p} + \lambda\bigg(\bigg\{v : -\frac{\mu}{p}\log(1 - v + ve^{-\frac{p}{\mu}}) \ge 1 - x\bigg\}\bigg)  \\ &= \frac{y(\mu)}{p} + \frac{e^{-\frac{p(1 - x)}{\mu}} - e^{-\frac{p}{\mu}}}{1 - e^{-\frac{p}{\mu}}} = \frac{e^{-\frac{p(1 - x)}{\mu}}}{1 - e^{-\frac{p}{\mu}}}.
  \end{align*}
  
  To show that $\mu\mapsto \rho(f^\mu)$ is a Devil's staircase, as in \cite{Lionel_parallel}, we need to show that $\mu \mapsto f^\mu$ is increasing, continuous with respect to the supremum norm, and that $(f^\mu)^n \neq \id_{\R} + k$ for each $n \ge 1$, $k \in \Z$. (Note that the last condition says that if $\overline{f}^\mu : \Sphere^1 \to \Sphere^1$ is the circle map corresponding to $f^\mu$ then $(\overline{f}^\mu)^n$ is not the identity.)
  
  To show that $\mu\mapsto f^\mu$ is increasing, we need to show for $x \in [0, 1]$, $\mu < \mu'$ that $f^\mu(x) \le f^{\mu'}(x)$. This inequality easily follows from $e^{-\frac{p(1 - x)}{\mu}} < e^{-\frac{p(1 - x)}{\mu}}$ and $1 - e^{-\frac{p}{\mu}} > 1 - e^{-\frac{p}{\mu'}}$. 
  
  Now we show that $\mu\mapsto f^\mu$ is continuous with respect to the supremum norm. For $\mu < \mu'$, 
  \begin{align*}
    \left|\frac{e^{-\frac{p(1 - x)}{\mu'}}}{1 - e^{-\frac{p}{\mu'}}} - \frac{e^{-\frac{p(1 - x)}{\mu}}}{1 - e^{-\frac{p}{\mu}}}\right| 
    \le \left|\frac{e^{-\frac{p(1 - x)}{\mu'}}}{1 - e^{-\frac{p}{\mu'}}} - \frac{e^{-\frac{p(1 - x)}{\mu'}}}{1 - e^{-\frac{p}{\mu}}}\right| + \left|\frac{e^{-\frac{p(1 - x)}{\mu'}}}{1 - e^{-\frac{p}{\mu}}} - \frac{e^{-\frac{p(1 - x)}{\mu}}}{1 - e^{-\frac{p}{\mu}}}\right| \\
    \le \left|\frac{1}{1 - e^{-\frac{p}{\mu'}}} - \frac{1}{1 - e^{-\frac{p}{\mu}}}\right|
      + \left|\frac{e^{-\frac{p(1 - x)}{\mu'}}}{1 - e^{-\frac{p}{\mu}}} - \frac{e^{-\frac{p(1 - x)}{\mu}}}{1 - e^{-\frac{p}{\mu}}}\right|,
  \end{align*}
  where, using the fact that for $x \ge 0$, $1 - e^{-x} \le x$, 
  \begin{align*}
    \left|\frac{e^{-\frac{p(1 - x)}{\mu'}}}{1 - e^{-\frac{p}{\mu}}} - \frac{e^{-\frac{p(1 - x)}{\mu}}}{1 - e^{-\frac{p}{\mu}}}\right| 
    \le \frac{1}{1 - e^{-\frac{p}{\mu}}}e^{-\frac{p(1 - x)}{\mu'}}\left(1 - e^{-\left(\frac{p(1 - x)}{\mu} - \frac{p(1 - x)}{\mu'}\right)}\right) \\ 
    \le \frac{1}{1 - e^{-\frac{p}{\mu}}}\left(\frac{p(1 - x)}{\mu} - \frac{p(1 - x)}{\mu'}\right) 
    = \frac{1}{1 - e^{-\frac{p}{\mu}}} \frac{(\mu - \mu')p (1 - x)}{\mu\mu'} \\ 
    \le \frac{1}{1 - e^{-\frac{p}{\mu}}} \frac{(\mu - \mu')p }{\mu\mu'}.
  \end{align*}
  Thus, 
  $$\|f^\mu - f^{\mu'}\| \le \left|\frac{1}{1 - e^{-\frac{p}{\mu'}}} - \frac{1}{1 - e^{-\frac{p}{\mu}}}\right| + \frac{1}{1 - e^{-\frac{p}{\mu}}} \frac{(\mu - \mu')p }{\mu\mu'},$$
  showing that $\mu\mapsto f^\mu$ is continuous.
  
  It remains to show that $(f^\mu)^n \neq \id_{\R} + k$ for any $n \ge 1$, $k \in \Z$. Let us fix $n \ge 1$, and choose $\varepsilon > 0$ small enough so that $(f^\mu)^k((0, \varepsilon))$ does not contain an integer point for any $k \le n$. To finish the proof of the proposition, we now show that $(f^\mu)^n$ is strictly convex on the interval $(0, \varepsilon)$. 
  
  For $x \in (0, 1)$, the derivative and second derivative of $f^\mu$ exists at $x$, and is positive, since
  $$
    (f^\mu)'(x) = \frac{p}{\mu}\cdot\frac{e^{-\frac{p(1 - x)}{\mu}}}{1 - e^{-\frac{p}{\mu}}}\quad\quad (f^\mu)''(x) = \left(\frac{p}{\mu}\right)^2\frac{e^{-\frac{p(1 - x)}{\mu}}}{1 - e^{-\frac{p}{\mu}}},
  $$
  and using the property $f^\mu(x + 1) = f^\mu(x) + 1$,
  \begin{align}
    \label{e:f''(x) > 0}
    (f^\mu)'(x) > 0 \text{ and } (f^\mu)''(x) > 0 \text{ for every } x \in \R \setminus \Z. 
  \end{align}
  Since the composition of twice differentiable functions is twice differentiable, $(f^\mu)^k$ is twice differentiable on $(0, \varepsilon)$. It is enough to show that $((f^\mu)^k)''(x) > 0$ for any $x \in (0, \varepsilon)$ and $k \le n$, which we prove by induction on $k$ together with the statement $((f^\mu)^k)'(x) > 0$. 
  
  For $k = 1$ the statements follows from \eqref{e:f''(x) > 0}. Now suppose that the statements are true for $k < n$, we wish to prove it for $k + 1$. By the choice of $\varepsilon$, $(f^\mu)^k((0, \varepsilon)) \subseteq (n, n + 1)$ for some $n \in \Z$, hence $f^\mu$ is twice differentiable on $(f^\mu)^k((0, \varepsilon))$ with a positive derivative and second derivative. Hence, $((f^\mu)^{k + 1})' = (f^\mu \circ (f^\mu)^k)' = ((f^\mu)'\circ (f^\mu)^k) \cdot ((f^\mu)^k)' > 0$ on $(0, \varepsilon)$ by the induction hypothesis and \eqref{e:f''(x) > 0}. Similarly, $((f^\mu)^{k + 1})'' = (f^\mu \circ (f^\mu)^k)'' = ((f^\mu)''\circ (f^\mu)^k) \cdot (((f^\mu)^k)')^2 + ((f^\mu)'\circ (f^\mu)^k) \cdot ((f^\mu)^k)'' > 0$, again using the induction hypothesis and \eqref{e:f''(x) > 0}. Thus the proof of the proposition is complete. 
\end{proof}

\begin{proof}[Proof of Lemma \ref{l:sigma_n_tends_to_sigma}]
Let $X_1^n, \dots X_n^n$ be independent Geometric random variables with mean $\mu n$, i.e., $X_i^n\sim Geometric(\frac{1}{1+\mu n})$ for all $i \le n$ so that $\sigma_n^\mu(v_i)$ is the $i^{th}$ smallest among $\{X_1^n,\dots, X_n^n\}$. Let $F_n:[0,\infty]\to \mathbb [0,1]$ be the appropriately normalized empirical distribution function, which in our case is $F_n(t)=\frac{1}{n} \sum_{k=1}^n I\{X^n_k\leq tn\}$, where we normalize by $n$ to match the graphon case. Let $E:[0,\infty] \to [0,1]$, $E(t)=1-e^{-\frac{t}{\mu}}$ which is the inverse of $\sigma^\mu$ taken as a function from $[0,1]$ to $\mathbb{R}_+$. Notice that the graph of $\sigma^\mu$ is the mirror image of the graph of $E$. Moreover, if we connect the points $(x,\lim_{y\to x^-}\tilde{\sigma}^\mu_n(x))$ and $(x,\lim_{y\to x^+}\tilde{\sigma}^\mu_n(x))$ for all jumping points in the graph of $\tilde{\sigma}^\mu_n$, and similarly for the graph of $F_n$, then the two obtained broken lines are once again mirror images of each other. Hence $\|\tilde{\sigma}^\mu_n - \sigma^\mu\|_1=\|F_n - E\|_1$.
Thus, it is enough to prove that $\|F_n - E\|_1 \to 0$ with probability $1$ as $n\to\infty$. 

Let $I^n_k(t)=I\{X^n_k\leq tn\}$. Then $F_n(t)=\frac{1}{n}\sum_{k=1}^n I^n_k(t)$. Let $F:[0,\infty]\to[0,1]$ be defined as 
\begin{align*}
    F(t) &= \frac{1}{n}\sum_{k=1}^n \mathbb{E}I_k^n(t)=1-\left(1-\frac{1}{1 + \mu n}\right)^{1 + \lfloor tn \rfloor}.
\end{align*}

Then $\|F_n - E\|_1\leq \|F_n - F\|_1 + \|F-E\|_1$. We first bound the term $\|F-E\|_1$.
\begin{align*}
&\|F-E\|_1=\int_0^{\infty} \left|1-\left(1-\frac{1}{1 + \mu n}\right)^{1 + \lfloor tn \rfloor} - \left(1-e^{-\frac{t}{\mu}}\right)\right| \dt \\ &=\int_0^{\infty}\left|e^{-\frac{t}{\mu}}-\left(1-\frac{1}{1 + \mu n}\right)^{1 + \lfloor tn \rfloor}\right| \dt \\ 
&\leq 
\int_0^{t_0}\left|e^{-\frac{t}{\mu}}-\left(1-\frac{1}{1 + \mu n}\right)^{1 + \lfloor tn \rfloor}\right| \dt + \int_{t_0}^{\infty}\left|e^{-\frac{t}{\mu}}-\left(1-\frac{1}{1 + \mu n}\right)^{1 + \lfloor tn \rfloor}\right| \dt.  
\end{align*}

Since $\int_0^{\infty}e^{-\frac{t}{\mu}} dt < \infty$, for any fixed $\varepsilon$, for large enough $t_0$, $\int_{t_0}^\infty|e^{-\frac{t}{\mu}}|dt < \varepsilon$.
Using that $\left(1-\frac{1}{1 + \mu n}\right)^{1 + \mu n}\leq \frac{1}{e}$ and that $\frac{1 + \lfloor tn \rfloor}{1 + \mu n} > \frac{t}{1 + \mu}$ for $n \ge 1$, it is clear that for large enough $t_0$, 
\begin{align*}
 \int_{t_0}^\infty \left(1-\frac{1}{1 + \mu n}\right)^{1 + \lfloor tn \rfloor} \dt= \int_{t_0}^\infty \left(\left(1-\frac{1}{1 + \mu n}\right)^{1 + \mu n}\right)^\frac{1 + \lfloor tn \rfloor}{1 + \mu n} \dt < \\
 \int_{t_0}^\infty \frac{1}{e^\frac{t}{1 + \mu}} \dt < \varepsilon
\end{align*}
for any $n \ge 1$. 
Let us fix a $t_0$ large enough so that both conditions are satisfied, then $\int_{t_0}^{\infty}\left|e^{-\frac{t}{\mu}}-\left(1-\frac{1}{1 + \mu n}\right)^{1 + \lfloor tn \rfloor}\right| \dt \leq 2\varepsilon$.

For this fixed $t_0$, 
\begin{align*}
\int_0^{t_0} & \left| e^{-\frac{t}{\mu}}-\left(1-\frac{1}{1 + \mu n}\right)^{1 + \lfloor tn \rfloor}\right| \dt \\
=& \int_0^{t_0}\left|e^{-\frac{t}{\mu}}-\left(\left(1-\frac{1}{1 + \mu n}\right)^{1 + \mu n}\right)^\frac{1 + \lfloor tn \rfloor}{1 + \mu n}\right| \dt \\ 
\leq & \int_0^{t_0}\left|e^{-\frac{t}{\mu}}-\left(\left(1-\frac{1}{1 + \mu n}\right)^{1 + \mu n}\right)^\frac{t}{\mu}\right| \dt \\ 
& + \int_0^{t_0}\left|\left(\left(1-\frac{1}{1 + \mu n}\right)^{1 + \mu n}\right)^\frac{t}{\mu}-\left(\left(1-\frac{1}{1 + \mu n}\right)^{1 + \mu n}\right)^\frac{1 + \lfloor tn \rfloor}{1 + \mu n}\right| \dt .  
\end{align*}
In the first term, $|e^{-\frac{t}{\mu}}-\big( (1-\frac{1}{1 + \mu n})^{1 + \mu n}\big)^{\frac{t}{\mu}}|$ is a  continuous function in $t$, and as $n$ increases, it monotonically tends to $0$ pointwise. Hence by the theorem of Dini, $|e^{-\frac{t}{\mu}}-\big( (1-\frac{1}{\mu n})^{\mu n}\big)^{\frac{t}{\mu}}|$ uniformly tends to the constant zero function as $n\to\infty$. Thus, for a large enough $n$, the first term is smaller than $\varepsilon$.

For the second term, \begin{align*}
  \int_0^{t_0}\left|\left(\left(1-\frac{1}{1 + \mu n}\right)^{1 + \mu n}\right)^\frac{t}{\mu}-\left(\left(1-\frac{1}{1 + \mu n}\right)^{1 + \mu n}\right)^\frac{1 + \lfloor tn \rfloor}{1 + \mu n}\right| \dt \\
  = \int_0^{t_0} \left(\left(1-\frac{1}{1 + \mu n}\right)^{1 + \mu n}\right)^\frac{t}{\mu} \left|\left(\left(1-\frac{1}{1 + \mu n}\right)^{1 + \mu n}\right)^\frac{\mu + \lfloor tn \rfloor \mu - t - tn\mu}{\mu(1 + \mu n)} - 1\right| \dt \\
  \le \int_0^{t_0} \left(\left(1-\frac{1}{1 + \mu n}\right)^{1 + \mu n}\right)^\frac{t}{\mu} \left(\left(\left(1-\frac{1}{1 + \mu n}\right)^{1 + \mu n}\right)^\frac{-t_0 - \mu}{\mu(1 + \mu n)} - 1\right) \dt,
\end{align*}
where the inequality comes from the fact that $(1 - \frac{1}{1 + \mu n})^{1 + \mu n}$ is always less than 1, so if the exponent, $\frac{\mu + \lfloor tn \rfloor \mu - t - tn\mu}{\mu(1 + \mu n)}$ is positive then multiplying the exponent by $-1$ and then decreasing it increases the distance of the expression from $1$. If the exponent is negative, then we simply decreased it, so the distance from $1$ increased in this case as well.
The second term of the last product clearly tends to $0$, hence the whole integral is at most $\varepsilon$ for large enough $n$. 

This means that for an arbitrary $\varepsilon$, if $n$ is large enough, then $\|F-E\|_1\leq 4\varepsilon$.

To bound the term $\|F_n - F\|_1$, we copy the standard proof of the Glivenko--Cantelli theorem.
Note that $|I^n_k(t) - \mathbb{E}I^n_k(t)|\leq 1$. Hence we can apply Azuma's inequality to get
$P(|F_n(t)-F(t)|>s)=P(|\sum_{k=1}^n (I^n_k(t) -\mathbb{E}I^n_k(t))|\geq ns)\leq 2e^{-\frac{ns^2}{2}}$ regardless of the value of $t$.  

Now take $t_0=0, t_1, \dots ,t_{m-1,} t_m = \infty$ such that $F(t_i)=\frac{i}{m}$. This can be done since $F$ is continuous, it is zero in 0 and tends to one in infinity. Now
$$P\left(\max_{i=1,\dots, m-1} \{|F_n(t_i)-F(t_i)|\}>s\right)\leq 2m\cdot e^{-\frac{ns^2}{2}}.$$

Take again an arbitrary $t\geq 0$. There exists some $i$ such that $t_i \leq t < t_{i+1}$. As $F_n$ and $F$ are both monotone increasing, $F_n(t_i)\leq F_n(t) \leq F_n(t_{i+1})$ and $F(t_i)\leq F(t) \leq F(t_{i+1})=F(t_i)+\frac{1}{m}$. Hence $F_n(t)-F(t)\leq F_n(t_{i+1})-F(t_i)=F_n(t_{i+1})-F(t_{i+1})+\frac{1}{m}$ and $F(t)-F_n(t)\leq F(t_{i+1})-F_n(t_i)=F(t_i)-F_n(t_i)+\frac{1}{m}$. Thus, for any $m$, and any $t$, $$\sup_{t\in \mathbb{R}}|F_n(t)-F(t)|\leq \max_{i=0,\dots m}|F_n(t_i)-F(t_i)|+\frac{1}{m}.$$ 

By choosing $s= \frac{\varepsilon}{2n^{1/3}}$ and $m= \frac{2n^{1/3}}{\varepsilon}$, we get
\begin{align*}
  P\left(\sup_{t\in \mathbb{R}} \left\{|F_n(t)-F(t)|\right\}>\frac{\varepsilon}{n^{1/3}} \right) &\leq P\left(\max_{i=0,\dots, m} \{|F_n(t_i)-F(t_i)|\}>\frac{\varepsilon}{2n^{1/3}}\right) \\ 
  &\leq 2\frac{2n^{1/3}}{\varepsilon}\cdot e^{-\frac{n^{1/3}\varepsilon^2}{8}}.    
\end{align*}

This implies that for a fixed $\varepsilon$, 
$\sum_{n=1}^{\infty}P(\sup_{t\in \mathbb{R}} \{|F_n(t)-F(t)|\}>\frac{\varepsilon}{n^{1/3}}) < \infty$, hence by the Borel--Cantelli lemma, with probability one, there exists $n_0\in \mathbb{N}$ such that for $n\geq n_0$, $\sup_{t\in \mathbb{R}} \{|F_n(t)-F(t)|\}\leq \frac{\varepsilon}{n^{1/3}}$. Repeating this argument for a series $\varepsilon_1, \varepsilon_2, \dots $ tending to zero, we get that with probability one, for each $\varepsilon > 0$, there exists $n_0\in \mathbb{N}$ such that for $n\geq n_0$, $\sup_{t\in \mathbb{R}} \{|F_n(t)-F(t)|\}\leq \frac{\varepsilon}{n^{1/3}}$.

We have proved that $F_n$ and $F$ are uniformly close to each other for large $n$ with high probability. Now we show that the integral of their difference is small for large values of $t$. Fix $\varepsilon_0 > 0$ small enough so that $e - \varepsilon_0 > 2$. Then  $P(X^n_k \ge \frac{n^{4/3}}{\sqrt{\varepsilon_0}})= (1-\frac{1}{1 + \mu n})^{\left\lceil\frac{n^{4/3}}{\sqrt{\varepsilon_0}}\right\rceil}$, hence
$$P\left(\max_k X^n_k \ge \frac{n^{4/3}}{\sqrt{\varepsilon_0}}\right)\leq n\left(1-\frac{1}{1+\mu n}\right)^{\frac{n^{4/3}}{\sqrt{\varepsilon_0}}}=n \left(\left(1-\frac{1}{1 + \mu n}\right)^{\mu n}\right)^{\frac{n^{1/3}}{\mu \sqrt{\varepsilon_0}}}.$$

For large enough $n$, $n\leq 2^{\frac{n^{1/3}}{\mu \sqrt{\varepsilon_0}}}$ and also $(1-\frac{1}{1 + \mu n})^{\mu n}\leq \frac{1}{e - \varepsilon_0}$. Hence for large enough $n$, 
$$P\left(\max_k X^n_k \ge \frac{n^{4/3}}{\sqrt{\varepsilon_0 }}\right)\leq \left(\frac{2}{e - \varepsilon_0}\right)^{\frac{n^{1/3}}{\mu \sqrt{\varepsilon_0}}}.$$ 
Since $\frac{2}{e - \varepsilon_0}<1$, this means that 
$$\sum_{n=1}^\infty P\left(\max_k X^n_k \ge \frac{n^{4/3}}{\sqrt{\varepsilon_0}}\right) < \infty.$$
Once again using the Borel--Cantelli lemma, with probability one, there exists $n_0\in \mathbb{N}$ such that for $n\geq n_0$, $\max_k X^n_k \leq \frac{n^{4/3}}{\sqrt{\varepsilon_0}}$. For $\varepsilon > 0$, if $\varepsilon < \varepsilon_0$ then  $\frac{n^{4/3}}{\sqrt{\varepsilon_0}} < \frac{n^{4/3}}{\sqrt{\varepsilon}}$, hence the above bound holds for each such $\varepsilon$.

Notice that $\sup\{t: F_n(\frac{t}{n})< 1 \}=\max \{X_k^n : k=1,\dots , n\}$. Hence with probability one, for each $\varepsilon > 0$, $\varepsilon < \varepsilon_0$ there exists $n_0\in \mathbb{N}$ such that for $n\geq n_0$, $F_n\left(\frac{n^{1/3}}{\sqrt{\varepsilon}}\right) = 1$.

Hence with probability one, for each $\varepsilon > 0$, $\varepsilon < \varepsilon_0$, there exists $n_0\in \mathbb{N}$ such that for $n\geq n_0$, 
\begin{align*}
\|F_n - F\|_1=\int_0^{\frac{n^{1/3}}{\sqrt{\varepsilon}}} |F_n(x) - F(x)| \dx + \int_{\frac{n^{1/3}}{\sqrt{\varepsilon}}}^\infty |1-F(x)| \dx \leq  \\
\frac{\varepsilon}{n^{1/3}} \cdot \frac{n^{1/3}}{\sqrt{\varepsilon}} + \int_{\frac{n^{1/3}}{\sqrt{\varepsilon}}}^\infty e^{-\frac{x}{\mu}} \dx = \\
\sqrt{\varepsilon} + (-\mu e^{-\frac{x}{\mu}})|_{x=\frac{n^{1/3}}{\sqrt{\varepsilon}}}^{x=\infty}= \sqrt{\varepsilon} + \mu e^{-\frac{n^{1/3}}{\mu\sqrt{\varepsilon}}}.
\end{align*}

This proves that $\|F_n - F\|_1 \to 0$ with probability $1$ as $n\to\infty$. 

Altogether, we obtain that $\|F_n - E\|_1 \to 0$ with probability $1$ as $n\to\infty$.
\end{proof}

\begin{proof}[Proof of Theorem \ref{t:geometric_staircase}]
Fix an arbitrary $\mu \in [0,1]$.
We would like to apply Theorem \ref{thm:robustness} to $C_p$ and $\sigma^\mu$.

We claim that $(C_p, \sigma^\mu)$ is s smooth pair. This can be proved analogously to the corresponding statement in the proof of Theorem \ref{t:ER_Devil}. Also, $C_p$ has finite diameter, as noted in the proof of Theorem \ref{t:ER_Devil}.

By Theorem \ref{t:random_graphs_conv}, Lemma \ref{l:sigma_n_tends_to_sigma} and Proposition \ref{prop:min_degree_in_random_graph_seq}, with probability one we can apply Theorem \ref{thm:robustness} to $(C_p, \sigma^\mu)$ to get that for any $\varepsilon > 0$, if $n$ is large enough, then $|a(C_p, \sigma^\mu)-a(G_n, \sigma^\mu_n)|\leq \varepsilon$.

Applying the above argument to a dense countable subset of $\mu$ values and a sequence of $\varepsilon$ values tending to zero, we get that with probability one, $a(G_n, \sigma^\mu_n)$ tends to $a(C_p, \sigma^\mu)$ for a dense set of $\mu$ values. Because of the way we coupled the random chip configuration $\sigma^\mu_n$, if we increase the value of $\mu$, then the number of chips monotonically increases on each vertex in each outcome. Hence on each outcome, $a(G_n,\sigma^\mu_n)$ monotonically increases if we increase $\mu$, using Lemma \ref{l:more_chips_fire_more}. $\sigma^\mu$ also increases pointwise in $\mu$, hence $a(C_p,\sigma^\mu)$ also increases monotonically. As $\mu \mapsto a(C_p, \sigma^\mu)$ is continuous, if $a(G_n, \sigma^\mu_n)$ tends to $a(C_p, \sigma^\mu)$ for a dense set of $\mu$ values, then $a(G_n, \sigma^\mu_n)$ tends to $a(C_p, \sigma^\mu)$ for each $\mu\in[0,1]$. We conclude that with probability one, $a(G_n, \sigma^\mu_n)$ tends to $a(C_p, \sigma^\mu)$ pointwise. As by Lemma \ref{l:limit_geom_devil's_staircase}, the map $\mu \mapsto a(C_p,\sigma^\mu)$ is a Devil's staircase, we obtained the statement of the Theorem.
\end{proof}

\appendix
\section{Basic properties of random graphs}\label{app:random_graphs}

Here we collect some well-known basic properties of random graphs. Throughout the section, $G(n,p)$ again means the random graph with $n$ vertices, where each edge is present independently with probability $p$. 

\begin{prop}\label{prop:min_degree_in_random_graph_seq}
Let $d<p$ be a fixed constant.
If $(G_n)_{n\in \mathbb{N}}$ is a sequence of random graphs where $G_n = G(n,p)$, then with probability one, there exists an index $n_0$ such that for each $n\geq n_0$, $mindeg(G_n)\geq dn$.
\end{prop}

We will use the following form of Azuma's inequality.
\begin{thm}[Azuma's inequality]
Suppose that $X_1,\dots, X_n$  are independent random variables, $\mathbb{E}[X_i]=0$ for each $i\in \mathbb{N}$, and for each $i$ there exist $c_i>0$ such that, $|X_i|\leq c_i$ almost surely. Then $$\mathbb{P}\left[\sum_{i=1}^n X_i > t \right]\leq e^{-\frac{t^2}{2\sum_{i=1}^n c^2_i}}.$$
\end{thm}

\begin{claim}\label{cl:azuma_for_degree}
For a vertex $v\in V(G_n)$, $\mathbb{P}[|\deg_{G_n}(v)-np|>\eta n]\leq 2e^{-\frac{n\eta^2}{2}}$. 
\end{claim}
\begin{proof}
We use Azuma's inequality with $X_u=\mathbf{1}_{\{uv \text{ is an edge}\}}-p$. Then $\{X_u\}_{u\in V \setminus \{v\}}$ is a set of independent random variables, $\mathbb{E}[X_u]=0$ and $|X_u|\leq \max\{p,1-p\}\leq 1$ for any $u\in V\setminus \{v\}$. Azuma's inequality applied for $\{X_u\}_{u \in V \setminus \{v\}}$ and for $\{-X_u\}_{u \in V \setminus \{v\}}$  gives us the above bound.
\end{proof}

\begin{proof}[Proof of Proposition \ref{prop:min_degree_in_random_graph_seq}]
Let $A_n$ be the event that $mindeg(G_n) < dn$. We need to show that the probability that infinitely many $A_n$'s occur is zero. 
By the Borel--Cantelli lemma, it suffices to show that $\sum_{n=1}^\infty \mathbb{P}(A_n) < \infty$.

\begin{align*}
\mathbb{P}(A_n) &=\mathbb{P}(\bigcup_{v\in V(G_n)}\{\deg_{G_n}(v)<dn\})\leq \sum_{v\in V(G_n)}\mathbb{P}(\deg_{G_n}(v)<dn)\\ &\leq \sum_{v\in V(G_n)}\mathbb{P}(|\deg_{G_n}(v)-np|>(p-d)n)\leq 2ne^{-\frac{n(p-d)^2}{2}},
\end{align*} 
where the last inequality follows from Claim \ref{cl:azuma_for_degree}.

Hence 
\begin{align*}
\sum_{n=1}^\infty \mathbb{P}(A_n)\leq  \sum_{n=1}^\infty 2ne^{-\frac{n(p-d)^2}{2}} < \infty.
\end{align*}
\end{proof}

\begin{prop}\label{prop:random_graph_connected}
If $(G_n)_{n\in \mathbb{N}}$ is a sequence of random graphs where $G_n = G(n,p)$, then with probability one, there exists an index $n_0$ such that for each $n\geq n_0$, $G_n$ is connected.
\end{prop}
\begin{proof}
This is a well-known fact; we include its short proof for completeness. We bound the probability that $G_n$ is disconnected.
If $G_n$ is disconnected, then there is a set $S$ of $k$ vertices for some $k\leq n/2$ such that no edge links $S$ to $S^c$. Hence one can bound
\begin{align*}
\mathbb{P}(G_n \text{ is disconnected})\leq \sum_{k=1}^{\lfloor \frac{n}{2}\rfloor}\binom{n}{k}q^{k(n-k)}\leq 
\sum_{k=1}^{\lfloor \frac{n}{2}\rfloor}n^k q^{k(n-k)}\leq \\
\sum_{k=1}^{\lfloor \frac{n}{2}\rfloor} (nq^{n-k})^k\leq 
\sum_{k=1}^{\lfloor \frac{n}{2}\rfloor} (nq^{\frac{n}{2}})^k\leq 
nq^{\frac{n}{2}} \cdot \frac{1-(nq^{\frac{n}{2}})^{\frac{n}{2}}}{1-nq^{\frac{n}{2}}},
\end{align*}
where $q= 1 - p$. For large enough $n$, $nq^{\frac{n}{2}} < \frac{1}{2}$, hence $\mathbb{P}(G_n$ is disconnected$) \le 2nq^{\frac{n}{2}}$ for large enough $n$. 
Hence $\sum_{n=1}^\infty \mathbb{P}(G_n \text{ is disconnected}) < \infty$. By the Borel--Cantelli lemma, we can conclude the statement of the proposition.
\end{proof}

\section*{Acknowledgment}
We would like to thank Swee Hong Chan for valuable discussions.


\begin{thebibliography}{9}
	\bibitem{Fekete}
	Fekete, M. \textit{Über die Verteilung der Wurzeln bei gewissen algebraischen Gleichungen mit ganzzahligen Koeffizienten}, Mathematische Zeitschrift. 17 (1): 228--249, 1923. doi:10.1007/BF01504345.
	\bibitem{bagnoli}
    F. Bagnoli, F. Cecconi, A. Flammini, and A. Vespignani, \textit{Short-period attractors and non-ergodic behavior in the deterministic fixed-energy sandpile model}, Europhys. Lett. 63 (2003), 512–518.
    \bibitem{BCLSV06}
    Borgs, C.; Chayes, J. T.; Lov\'asz, L.; S\'os, V. T.; Vesztergombi, K. Counting graph homomorphisms,
in: \emph{Topics in Discrete Mathematics} (ed. M. Klazar, J. Kratochvil, M. Loebl, J. Matoušek,
R. Thomas, P. Valtr), Springer (2006), 315--371.
    \bibitem{BCLSV08} 
    Borgs, C.; Chayes, J. T.; Lov\'asz, L.; S\'os, V. T.; Vesztergombi, K. Convergent sequences of dense graphs. I. Subgraph frequencies, metric properties and testing. \emph{Adv. Math.} 219 (2008), no. 6, 1801--1851.
    \bibitem{Markov chain}
    Roberts, Gareth O.; Rosenthal, Jeffrey S. General state space Markov chains and MCMC algorithms. \emph{Probab. Surv.} 1 (2004), 20--71.
    \bibitem{Doeblin}
    Doeblin, W. El\'ements d'une th\'eorie g\'en\'erale des cha\^{\i}nes simples constantes de Markoff. \emph{Annales Scientifiques de l'Ecole Normale Sup\'erieure}, Paris, III Ser., 57 (1940), 61--111.
	\bibitem{K}
    Kechris, Alexander S. \emph{Classical descriptive set theory.} Graduate Texts in Mathematics, 156. Springer-Verlag, New York, 1995.
    \bibitem{Lionel_parallel}
    Lionel Levine. \emph{Parallel chip-firing on the complete graph: devil's staircase and Poincaré rotation number} Ergodic Theory and Dynamical Systems (2011) 31: 891--910
    \bibitem{Lovasz_large_graphs}
    Lov\'asz, L\'aszl\'o. \emph{Large networks and graph limits.} American Mathematical Society Colloquium Publications, 60. American Mathematical Society, Providence, RI, 2012.
    \bibitem{Lovasz-Szegedy}
    Lov\'asz, L\'aszl\'o; Szegedy, Balázs. Limits of dense graph sequences. \emph{J. Combin. Theory Ser. B} 96 (2006), no. 6, 933--957.

\end{thebibliography}
\end{document}